\documentclass[11pt]{article}
\usepackage{amssymb,amsmath,amsfonts,amsthm,mathtools,enumerate,color}
\usepackage{mathrsfs,amsfonts}
\usepackage[toc,page,title,titletoc,header]{appendix}
\usepackage{graphicx,subfig}

\usepackage{paralist}

\usepackage{tikz}
\usetikzlibrary{arrows,positioning,shapes.geometric}

\graphicspath{ {./figures/} }

\usepackage{indentfirst}
\usepackage{multicol}
\usepackage{booktabs}
\usepackage{url}
\usepackage[outdir=./]{epstopdf} 

\usepackage{hyperref}
\hypersetup{
    colorlinks=true, 
    linktoc=all,     
    linkcolor=blue,  
}
\numberwithin{equation}{section}

\newtheorem{Theorem}{Theorem}[section]
\newtheorem{Lemma}[Theorem]{Lemma}
\newtheorem{Proposition}[Theorem]{Proposition}
\newtheorem{Corollary}[Theorem]{Corollary}

\newtheorem{Assumption}{H.\!\!}

\theoremstyle{definition}
\newtheorem{Definition}{Definition}[section]

\theoremstyle{remark}
\newtheorem{Remark}{Remark}[section]

 \def\p{\partial} \def\nb{\nonumber}
\def \Vh0{\stackrel{\circ}{V}_h} \def\to{\rightarrow}
     
\def\Om{\Omega}  \def\om{\omega} 
 
\newcommand{\q}{\quad}   \newcommand{\qq}{\qquad}

\def\l{\label}  \def\f{\frac}  \def\fa{\forall}
\def\b{\beta}  \def\a{\alpha} 
 
\def\eps{\varepsilon}

 \def\t{\times}  
\def\ms{\medskip}  

\def\p{\partial}
\def \la{\langle} \def\ra{\rangle}

\def\cA{\mathcal{A}}
\def\cB{\mathcal{B}}

\def\cD{\mathcal{D}}
\def\cE{\mathcal{E}}
\def\cF{\mathcal{F}}

\def\cJ{\mathcal{J}}

\def\cL{\mathcal{L}}

\def\cN{\mathcal{N}}
\def\cO{\mathcal{O}}
\def\cP{\mathcal{P}}

\def\cR{\mathcal{R}}

\def\cT{\mathcal{T}}
\def\cU{\mathcal{U}}
\def\cV{\mathcal{V}}
\def\cW{\mathcal{W}}

\def\sC{\mathscr{C}}

\def\bA{{\textbf{A}}}
\def\bB{{\textbf{B}}}

\def\N{{\mathbb{N}}}
\def\bP{\mathbb{P}}

\def\R{{\mathbb R}}

\newcommand{\ex}{\mathbb{E}}

\newcommand{\dimin}{\textnormal{dim}_\textnormal{in}}
\newcommand{\dimout}{\textnormal{dim}_\textnormal{out}}

\newcommand{\lc}
{\mathrel{\raise2pt\hbox{${\mathop<\limits_{\raise1pt\hbox
{\mbox{$\sim$}}}}$}}}

\newcommand{\gc}
{\mathrel{\raise2pt\hbox{${\mathop>\limits_{\raise1pt\hbox{\mbox{$\sim$}}}}$}}}

\newcommand{\ec}
{\mathrel{\raise2pt\hbox{${\mathop=\limits_{\raise1pt\hbox{\mbox{$\sim$}}}}$}}}

\def\bb{\begin{equation}} \def\ee{\end{equation}}
\def\bbn{\begin{equation*}} \def\een{\end{equation*}}

\def\beqn{\begin{eqnarray}}  \def\eqn{\end{eqnarray}}

\def\beqnx{\begin{eqnarray*}} \def\eqnx{\end{eqnarray*}}

\def\bn{\begin{enumerate}} \def\en{\end{enumerate}}

\def\bd{\begin{description}} \def\ed{\end{description}}



\begin{document}

%
\title{Rectified deep neural networks overcome the curse of dimensionality for nonsmooth value functions in  zero-sum games of nonlinear stiff systems}
\author{
Christoph Reisinger\thanks{Mathematical Institute, University of Oxford, United Kingdom ({\tt christoph.reisinger@maths.ox.ac.uk, yufei.zhang@maths.ox.ac.uk})}
\and
Yufei Zhang\footnotemark[2]
}
\date{}

\maketitle


\noindent\textbf{Abstract.} 
In this paper,  we establish that for  a wide class of   controlled stochastic differential equations (SDEs) with stiff coefficients, the   value functions of  corresponding zero-sum games can be represented by a deep artificial neural network (DNN), whose complexity  grows at most polynomially  in both the dimension of the state equation and the reciprocal of the required accuracy. 
Such nonlinear stiff systems may arise, for example, from Galerkin approximations of controlled  stochastic partial differential equations  (SPDEs), or controlled  PDEs with uncertain initial conditions and source terms.
This implies that  DNNs can break the curse of dimensionality in numerical approximations and optimal control of PDEs and SPDEs. 
The main ingredient of our proof  is to construct a suitable  discrete-time system to 
effectively approximate the evolution of the underlying stochastic dynamics.
Similar ideas can also be applied to obtain  expression rates of DNNs for value functions induced by stiff systems with regime switching coefficients and driven by general L\'{e}vy noise.


\ms

%
%

\medskip
 

\section{Introduction}\l{sec:intro}

In  this paper, we study the expressive power of  deep artificial neural networks (DNNs), and demonstrate that one can construct  
DNNs with polynomial complexity to approximate  nonsmooth value functions associated with  stiff stochastic differential equations (SDEs).

More precisely, for each $d\in \N$, we consider the value function  $v_d:\R^d\to \R$ of the following  $d$-dimensional zero-sum stochastic differential game on a finite time horizon $[0,T]$: 
$$v_d(x)\coloneqq \inf_{u_1\in \cU_{1,d}}\sup_{u_2\in \cU_{2,d}}\ex\bigg[f_d(Y^{x,d,u_1,u_2}_T)+g_{d}(u_1,u_2)\bigg], \q x\in \R^d,$$
where $\cU_{i,d}$, $i=1,2$, are sets of admissible open-loop control strategies (see Section \ref{sec:sde_control} for a precise definition), $f_d:\R^d\to \R$ is a (possibly nonsmooth) terminal cost function with at most quadratic growth at infinity, and  
for each $x\in \R^d$, $u_i\in \cU_{i,d}$, $i=1,2$, $(Y^{x,d,u_1,u_2}_t)_{t\in [0,T]}$ is the  solution to the following $d$-dimensional controlled SDE:
\bb \l{eq:sde_intro}
 dY_t=(-A_dY_t+\mu_d(t,Y_{t},u_1,u_2))\,dt+ {\sigma}_d(t,Y_{t},u_1,u_2)\,dB_t, \q t\in (0,T];\q Y_0=x,
\ee
where $A_d$ is a $d\t d$ matrix, $\mu_d$ and $\sigma_d$ are  respectively $\R^d$ and $\R^{d\t d}$-valued functions, and 
  $(B_t)_{t\in [0,T]}$ is a $d$-dimensional Brownian motion defined on a probability space $(\Om, \cF , \bP )$. 
  
 The above problem is called a zero-sum stochastic differential  game
 since the underlying SDE \eqref{eq:sde_intro}
is  controlled by two players with opposite objectives, i.e.,
the ``inf-player'' aims to minimize the associated cost function over all strategies $u_1\in\cU_{1,d}$, while the ``sup-player'' aims to maximize the same cost function over all strategies $u_2\in\cU_{2,d}$.
The admissible controls  $\cU_{i,d}$, $i=1,2$, are called open-loop controls
since they are  deterministic processes; see  page 23 of \cite{touzi2012} for different types of  strategies.
\color{black}
In the  case with $\sigma_d\equiv 0$,  \eqref{eq:sde_intro} degenerates to a controlled ordinary differential equation. Moreover, if one of the sets  $\cU_{1,d}$ and $ \cU_{2,d}$ is singleton, the zero-sum game reduces to an optimal control problem.

In this work, we shall  allow the coefficients $A_d$, $\mu_d$ and ${\sigma}_d$ to be stiff in the sense that they are Lipschitz continuous (with respect to  the Euclidean norm on $\R^d$) but the  Lipschitz constants grow polynomially in the dimension $d$. 
Such stiff SDEs arise naturally from spatial discretizations of stochastic partial differential equations (SPDEs) by using spectral methods (see e.g.~\cite{ito1996,gyongy2008,jentzen2011,luo2012,luo2015}),  or finite difference/element methods (see e.g.~\cite{gyongy2008,barth2012,giles2012}). 

For simplicity,  let us   consider  the following uncontrolled SPDE as  motivating example, but  similar arguments also apply to controlled SPDEs.
Let $B=(B(t))_{t\ge 0}$ be  an $m$-dimensional Brownian motion on a probability space $(\Om, \cF,\bP)$, 
$H=L^2(\R^{p})$, $V=H^1(\R^{p})$, and $V^*$ denotes the strong dual space of $V$. Here $H^*$ is identified with $H$ so that $V\subset H=H^*\subset V^*$. Then, for given mappings $A:V\mapsto V^*$, $U:[0,T]\t V\mapsto V^*$, and $G:[0,T]\t V\mapsto H^{m}$, it has been shown in \cite{ito1996,gyongy2008,jentzen2011} that the following semilinear SPDE (with a spatial domain in $\R^p$):
\bb\l{eq:spde}
dy(t)+Ay(t)\,dt=U(t,y(t))\,dt+G(t,y(t))\,dB(t), \; t\in (0,T],\q y(0)=y_0 \in H,
\ee
admits a  solution $y$
under the following strong monotonicity condition\footnotemark: 
\footnotetext{
In general,  the coefficients $U$ and $G$ need to satisfy other technical assumptions, such as continuity, coercivity and  growth conditions,
to ensure the well-posedness of \eqref{eq:spde} in $L^2(\Om\t [0,T];V)$; see e.g. \cite{gyongy2008}.
However, since we only use \eqref{eq:spde}  to motivate the high-dimensional stiff SDE \eqref{eq:spde_d}
and shall establish approximation results for the corresponding high-dimensional value functions,
we omit other technical assumptions on the coefficients $U$ and $G$ here
and introduce the precise conditions for the finite-dimensional SDEs in Sections \ref{sec:sde_stiff}  and \ref{sec:sde_control}.
\color{black}
}
there exist some $\lambda,\b>0$, such that for all $t\in [0,T]$, $u,v\in V$,
\bb\l{eq:coercive}
2\la u-v, -A(u-v)+U(t,u)-U(t,v)\ra_{V\t V^*} +\|G(t,u)-G(t,v)\|_H^2\le -\lambda \|u-v\|_V^2+ \b\|u-v\|^2_H, 
\ee
where $\la \cdot,\cdot\ra_{V\t V^*}$ denotes the duality product of $V\t V^*$ (see e.g.~Assumption 2.1(i) in \cite{gyongy2008}). 
Important special cases of \eqref{eq:spde} include suitable semilinear parabolic PDEs with (additive or multiplicative) noise
and the Zakai equation from nonlinear filtering (see e.g.~\cite{ito1996}) or from a large pool limit of interacting particles (see e.g.~\cite{kurtz1999, giles2012}). 
\color{black}

We are interested in the  value functional associated with the SPDE \eqref{eq:spde}:
\bb\l{eq:value_functional}
\cV:y(0)\in H\mapsto \ex[f(y(T))]\in \R,
\ee
where $y(0)$ is taken within a neighbourhood of the initial condition $y_0$ in \eqref{eq:spde}, and $f:V\to \R$ is a given locally Lipschitz cost functional. This is practically important if  the exact dynamics of \eqref{eq:spde} is only   known subject to uncertain initial conditions, or if we would like to compare the optimal cost of a control problem among all initial states (see e.g.~\cite{farhood2008,hu2015}). An accurate representation of the value functional is also crucial for the control design in reinforcement learning (see \cite{bertsekas1996}).

In practice, 
we shall consider a finite-dimensional version of
 the  (infinite-dimensional) value functional  $\cV$.
Let $\{e_k\}_{k\ge 1}$ be an orthonormal basis of $H$, made of elements in $V$, and $H_d=\textrm{span}\{e_k\mid k=1,\ldots,d\}$ for all $d\in \N$.
Then for each $d\ge 1$, we can  project the SPDE \eqref{eq:spde} onto the subspace $H_d$ and consider a  $d$-dimensional It\^{o}-Galerkin approximation of \eqref{eq:spde} in $\R^d$ of the form:
\bb\l{eq:spde_d}
dy_d(t)+A_dy_d(t)\,dt=U_d(t,y_d(t))\,dt+G_d(t,y_d(t))\,dB(t), \; t\in (0,T],\q y_d(0)=y_{0,d},
\ee
where the discrete operators $A_d, U_d, G_d$ satisfy a  monotonicity condition similar to \eqref{eq:coercive}.
Then, under suitable regularity assumptions,  one can show the well-posedness of a solution $y_d$ to the finite-dimensional SDE \eqref{eq:spde_d}, and estimate the rate of convergence in terms of the dimension $d$. 
The convergence  of $y_d(T)$ to $y(T)$ as $d\to\infty$ suggests us to  approximate the functional $\cV$ by the $d$-dimensional value function 
$$v_d:y_d(0)\in H_d\mapsto \ex[f(y_d(T))]\in \R$$ 
with a sufficiently large $d\in \N$. 
Note that for SPDEs driven by  $H$-valued random fields, one can consider similar It\^{o}-Galerkin SDEs with finite-dimensional noises by truncating the series representation of the (space-time) random process (see \cite{barth2012} for sufficient conditions under which this extra approximation of the noise preserves the overall convergence order in $d$).

%
However, we face  several difficulties  in approximating the $d$-dimensional value function $v_d$. 
Recall that the errors of the  Galerkin approximations 
for the SPDE \eqref{eq:spde} (with a spatial domain in $\R^p$)
are in general of the magnitude $\cO(d^{-\gamma/p})$ for some $\gamma>0$ (see e.g.~\cite{gyongy2008,barth2012,giles2012}).
Thus,
the local Lipschitz continuity of the cost function in \eqref{eq:value_functional} suggests that,
 to achieve an  accuracy $ \eps$ 
in representing the value functional $\cV:H=L^2(\R^p)\to \R$,
we need to  approximate the value function $v_d:H_d\to \R$,
where $H_d$ can be identified as a $d$-dimensional Euclidean space  with $d=\cO(\eps^{-p/\gamma})$.
Since 
many classical function approximation methods, such as piecewise constant and piecewise linear approximations, 
require a complexity of $\cO(\eps^{-d})$
to approximate a $d$-dimensional  function within an  accuracy $\eps$,
we see 
the total complexity 
for  classical methods
to represent the value functional \eqref{eq:value_functional}
within the accuracy $\eps$
is of the magnitude
$\cO(\eps^{-\eps^{-p/\gamma}})$, which suffers from the so-called Bellman's curse of dimensionality. 
\color{black}

Moreover, the control processes and the nonsmoothness of the terminal costs imply that  the value function  $v_d$ typically has   weak regularity, e.g.~$v_d$ is merely locally Lipschitz  continuous and could grow quadratically at infinity. This prevents us from approximating the value function by using 
sparse grid approximations \cite{bungartz2005,schwab2011}, or high-order polynomial  expansions \cite{kalise2018}. 
Finally, since the mappings $A$, $U$ and $G$ in \eqref{eq:spde} could involve differential operators,  the Lipschitz constants (with respect to the Euclidean norm) of $A_d, U_d, G_d$  in \eqref{eq:spde_d} 
will in general grow polynomially in dimension $d$. This stiffness of coefficients creates a difficulty in constructing efficient discrete-time dynamics to approximate  the  time evolution of the It\^{o}-Galerkin SDE \eqref{eq:spde_d}.

In recent years,   DNNs 
have  achieved remarkable performance in representing high-dimensional mappings  in  a wide range of applications (see e.g.~\cite{luo2012,lecun2015,luo2015,mnih2015,han2016,beck2017,hure2019,ito2019,pham2019} and the  references therein for  applications in  optimal control and numerical simulation of PDEs), and it seems that DNNs admit the flexibility to overcome the curse of dimensionality. However, even though there is a vast literature on the approximation theory of artificial neural networks 
(see e.g.~\cite{hornik1989,mhaskar2016,montanelli2017,petersen2017,yarotsky2017,arora2018,e2018,elbrachter2018,he2018,
hutzenthaler2018,
jentzen2018,schwab2018,bolcskei2019,hutzenthaler2019gradient,hutzenthaler2019,grohs2019,guhring2019,opschoor2019}), 
to  the best of our knowledge, 
only \cite{elbrachter2018,grohs2018,hutzenthaler2018,jentzen2018,hutzenthaler2019gradient,hutzenthaler2019} established DNNs' expression rates for approximating nonsmooth value functions 
(associated with  $d$-dimensional SDEs 
whose  diffusion coefficients
are affine with respect to the state variable
and 
both the drift and diffusion coefficients
are Lipschitz continuous with a constant \textit{independent of} the dimension $d$).

In this work, we shall extend their results  by giving a rigorous proof of the fact that DNNs do overcome the curse of dimensionality for approximating  (nonsmooth) value functions of zero-sum games of \textit{controlled} SDEs  with \textit{stiff, time-inhomogeneous, nonlinear}  coefficients. 
%
%
%
More precisely, we shall establish that for  a wide class of   controlled stiff SDEs, to represent the corresponding value functions with accuracy $\eps$,  the number of parameters in the employed DNNs  grows at most polynomially 
 in both the dimension of the state equation and the reciprocal of the accuracy $\eps$ (see Theorems \ref{thm:expression_sde} and  \ref{thm:expression_sde_ctrl}). 
As a direct consequence of these expression rates, we show   that one can approximate the viscosity solution to a Kolmogorov backward PDE with stiff coefficients by DNNs with polynomial complexity (see Corollary \ref{cor:pde}). In particular, if one further assumes that the Galerkin approximation of a controlled SPDE  has a convergence rate $\cO(d^{-\gamma})$ for some $\gamma>0$, our result indicates that we can represent the nonlinear value functional $\cV$ 
without the curse of dimensionality.

The approach  we take here is to first describe the evolution of a $d$-dimensional controlled SDE \eqref{eq:sde_intro} by using a suitable discrete-time dynamical system, and then constructing the desired DNN by a specific realization of the discrete-time dynamics. 
This is of the same spirit as \cite{jentzen2018}, where the authors represent an uncontrolled SDE with constant diffusion and nonlinear drift coefficients by its explicit Euler discretization. However, 
due to the stiffness of  the It\^{o}-Galerkin SDEs considered in this paper,  such an explicit time discretization will in fact lead to an approximation error depending exponentially on the dimension $d$ (cf.~\cite[Proposition 4.4]{jentzen2018}), and hence it cannot be used in our construction. We shall overcome this difficulty by 
approximating the underlying dynamics with its partial-implicit Euler discretization,  whose  error depends polynomially on the dimension $d$ and the (time) stepsize. We also adopt a two-step approximation of the terminal cost function involving truncation and extrapolation, which allows us to construct rectified neural networks for quadratically growing terminal costs; see  the discussion below  (H.\ref{assum:coeff_d}) for details.

The rest of this paper is structured as follows. Section \ref{sec:main} states the  assumptions and 
presents the main theoretical results of the expression rates. 
We discuss  several fundamental operations of DNNs in Section \ref{sec:cal}, and analyze a perturbed linear-implicit Euler discretization of SDEs in Section \ref{sec:im_euler}. Based on these estimates, we establish the expression rates of rectified neural networks for uncontrolled systems in Section \ref{sec:proof_expression_sde}, and controlled systems in Section \ref{sec:proof_expression_sde_ctrl}. 
Section \ref{sec:conclusion} offers possible extensions and directions for further research.

\section{Main results}\l{sec:main}

In this section, we shall recall the  notion of DNN,  and state our main results on the expression rates of DNNs for approximating value functions associated with controlled SDEs with stiff coefficients. 

We start with some  notation which is needed frequently throughout this work. For any given $d\in \N$, 
we denote by $\|\cdot\|$ the Euclidean norm of a vector in $\R^d$,  by $\la \cdot, \cdot\ra$ the  canonical Euclidean inner product, and by $I_d$ the $d\t d$ identity matrix.  
For a given matrix $A\in \R^{d_1\t d_2}$, we denote by $\|A\|$  the  Frobenius  norm of $A$, and by $\|A\|_{\textrm{op}}$
 the matrix norm induced by  Euclidean vector norms. 
We shall also denote by $C$ a generic constant, which may take a different value at each occurrence. Dependence of $C$ on parameters will be indicated explicitly by $C_{(\cdot)}$, e.g.~$C_{(\a,\b)}$.

Now we introduce the basic concepts of DNNs. By following the notation  in \cite{petersen2017,elbrachter2018,grohs2018} (up to some minor changes), we shall distinguish between a deep artificial neural network, represented as a structured set of weights, and its realization, a multi-valued function on $\R^d$.
This enables us to  construct complex neural networks from simple ones in an explicit and unambiguous way, and further analyze the complexity of DNNs.

\begin{Definition}[Deep artificial neural networks]\l{def:DNN}
Let $\cN$ be the set of DNNs given by
$$
\cN=\bigcup_{L\in \N} \bigcup_{(N_0,N_1,\ldots, N_L)\in \N^{L+1}} \cN^{N_0,N_1,\ldots, N_L}_L, \q \textnormal{where $\cN^{N_0,N_1,\ldots, N_L}_L=\bigtimes_{l=1}^L (\R^{N_l\t N_{l-1}}\t \R^{N_l})$}.
$$
Let $\sC,\cL,\dimin,\dimout:\cN\to \N$ and $\dim:\cN\to \cup_{L\in \N}\N^{L+1}$ be  functions 
such that for any given $\phi\in \cN^{N_0,N_1,\ldots, N_L}_L$, 
we have  $\sC(\phi)=\sum_{l=1}^L N_l(N_{l-1}+1)$, $\cL(\phi)=L$, $\dimin(\phi)=N_0$, $\dimout(\phi)=N_L$ and $\dim(\phi)=(N_0,N_1,\ldots, N_L)$.
We shall refer to the quantities $\sC(\phi)$, $\cL(\phi)$, $\dimin(\phi)$ and $\dimout(\phi)$  as the size, depth, input dimension and  output dimension  of the DNN $\phi$, respectively. 

For any given activation function $\varrho\in C(\R;\R)$, let $\varrho^*:\cup_{d\in \N}\R^d\to \cup_{d\in \N}\R^d$ be the function which satisfies for all $x=(x_1,\ldots, x_d)\in \R^d$ that $\varrho^*(x)=(\varrho^*(x_1),\ldots, \varrho^*(x_d))$, and let $\cR_\varrho:\cN\to \cup_{a,b\in \N}C(\R^a;\R^b)$ be the realization operator such that  for any given  $x_0\in \R^{N_0}$ and 
$$\phi=((W_1,b_1),(W_2,b_2),\ldots, (W_L,b_L))\in \cN^{N_0,N_1,\ldots, N_L}_L,
\q \textnormal{with $L\in  \N$ and $(N_0,N_1,\ldots, N_L)\in \N^{L+1}$},
$$
we have $\cR_{\varrho}(\phi)\in C(\R^{N_0}; \R^{N_L})$  defined recursively as follows: 
let $x_l=\varrho^*(W_lx_{l-1}+b_l)$ for all $l=1,\ldots, L-1$, and let 
$$
[\cR_{\varrho}(\phi)](x_0)=W_Lx_{L-1}+b_L.
$$

\end{Definition}

Roughly speaking, one can 
describe a DNN by its \textit{architecture}, that is the number of layers $L$ and the dimensions of all layers $N_0,N_1,\ldots, N_L$, together with the coefficients of the affine functions used to compute each layer from the previous one. Note that Definition \ref{def:DNN} does not specify a fixed nonlinear activation function in the architecture of a DNN, but instead considers the realization of a DNN with respect to a given activation function, which  allows us to study the approximation capacity of  DNNs with arbitrary activation functions (see e.g.~Lemma \ref{lemma:combination}). 

To simplify the presentation,
in the work we shall mainly focus on DNNs with the commonly used Rectified Linear Unit (ReLU) activation function, i.e., $\varrho(x)=\max(0,x)$, due to its representation flexibility.
Moreover, 
we  allow the weights of a DNN (i.e., the coefficients of the affine functions) to take arbitrary real numbers
when approximating a given function. 
A similar analysis can be carried out for networks with quantization
(i.e., the maximal magnitude of weights in the network is \textit{a priori} fixed), 
by allowing the \textit{a priori} bound of
the weights to 
increase in a controlled way (see e.g.~\cite{petersen2017}).
\color{black}

For any given DNN $\phi\in \cN$, the quantity $\sC(\phi)\in \N$ represents  the number of all real parameters, including zeros, used to describe the DNN. 
We remark that one can also consider the number of non-zero entries of the DNN  $\phi$ as in \cite{elbrachter2018}.
However, since it is in general difficult to build a sparse architecture with pre-allocated zero entries to approximate a desired value function, we choose to adopt the notation of `size' by considering all  parameters and quantify the complexity of the DNN  in a conservative manner.

Motivated by the  application to optimal control problems of SPDEs, in the remaining part of this section, we shall construct a sequence of DNNs $(\psi_{\eps,d})_{\eps,d}$, such that for each $\eps\in (0,1)$, $d\in \N$, $\psi_{\eps,d}$ represents  the value function $v_d$ induced by a $d$-dimensional stiff SDE with the accuracy $\eps$ on $\R^d$. We shall demonstrate that under a monotonicity condition similar to \eqref{eq:coercive}, the complexity of the constructed  DNN $\psi_{\eps,d}$ depends  polynomially on both $d$ and $\eps^{-1}$, i.e., the DNNs $(\psi_{\eps,d})_{\eps,d}$ overcome the curse of dimensionality.  We first give the results for uncontrolled SDEs with stiff coefficients in Section \ref{sec:sde_stiff}, and then extend the results to controlled SDEs with piecewise-constant strategies in Section \ref{sec:sde_control}.

\subsection{Expression rate for  SDEs and Kolmogorov PDEs with stiff coefficients}\l{sec:sde_stiff}
In this section, we present  the expression rate of DNNs for approximating value functions induced by nonlinear SDEs with stiff coefficients. 

We start by introducing the  value functions of interest. 
For each $d\in \N$, we consider the following value function:  
\bb\l{eq:value_unctrl}
v_d:x\in \R^d\mapsto \ex[f_d(Y^{x,d}_T)]\in \R,
\ee
where $Y^{x,d}=(Y^{x,d}_t)_{t\in [0,T]}$ is the strong solution to the following $d$-dimensional SDE:
 \bb\l{eq:sde_thm_d}
 dY^{x,d}_t=(-A_dY^{x,d}_t+\mu_d(t,Y^{x,d}_{t}))\,dt+ {\sigma}_d(t,Y^{x,d}_{t})\,dB_t, \q t\in (0,T];\q Y^{x,d}_0=x,
\ee
 with a $d$-dimensional Brownian motion  $(B_t)_{t\in [0,T]}$ defined on a probability space $(\Om, \cF , \bP )$. 

We now list the main assumptions on the coefficients.
\begin{Assumption}\l{assum:coeff_d}
Let $\b,\kappa_0,T\ge 0$ and $\eta>0$ be fixed constants. For all $d\in \N$ and $D>0$, let $A_d\in \R^{d\t d}$, and ${\mu}_d:[0,T]\t \R^d\to \R^d$, ${\sigma}_d:[0,T]\t \R^d\to \R^{d\t d}$, $f_d, f_{d,D}:\R^d\to \R$ be measurable functions 
satisfying the following conditions: 
\bn[(a)]
\item \l{assum:mono_d}
The matrix $A_d$ and the functions $(\mu_d,\sigma_d)$ satisfy 
for all $t\in [0,T]$ and $x,y\in \R^d$ that:
\begin{align}\l{eq:mono_d}
\la x-y,  {\mu}_d(t,x)-{\mu}_d(t,y)\ra +\eta\|{\mu}_d(t,x)-{\mu}_d(t,y)\|^2+&\f{1+\eta}{2}\|{\sigma}_d(t,x)-{\sigma}_d(t,y)\|^2\\
&\le \b \|x-y\|^2+ \la x-y,A_d(x-y)\ra. \nb
\end{align}
\item \l{assum:A_d}
$\|A_d\|_{\rm{op}}\le \kappa_0 d^{\kappa_0}$, and  $\la x,A_d x\ra\ge 0$ for all $x\in \R^d$. 
\item \l{assum:regu_d}
There exist constants $C^{\mu}_{d,0},C^{\mu}_{d,1},C^{\sigma}_{d,0},C^{\sigma}_{d,1}\in [0, \kappa_0 d^{\kappa_0}]$
such that 
for all $t,s\in [0,T]$, $x,y\in \R^d$
it holds that 
\begin{align*}
\|\mu_d(t,x)-\mu_d(s,y)\|\le C^{\mu}_{d,1}(\sqrt{t-s}+\|x-y\|),\q \|\mu_d(t,0)\|\le C^{\mu}_{d,0};\\
\|\sigma_d(t,x)-\sigma_d(s,y)\|\le C^{\sigma}_{d,1}(\sqrt{t-s}+\|x-y\|),\q \|\sigma_d(t,0)\|\le  C^{\sigma}_{d,0}.
\end{align*}
\item \l{assum:terminal_d}
There exists  a constant $C^f_d\in [0,\kappa_0 d^{\kappa_0}]$ 
such that 
for all $x\in \R^d$ and $D>0$ it holds that 
\begin{align*}
|f_{d}(0)|\le C^{f}_{d},\; |f_d(x)-f_{d,D}(x)|\le C^{f}_{d} \|x\|^{2}1_{B_\infty(D)^c}(x), \q |f_{d,D}(x)-f_{d,D}(y)|\le C^{f}_{d} D \|x-y\|, 
\end{align*}
where $B_\infty(D)\coloneqq \{x\in \R^d\mid x_i\in [-D,D],\, \fa i\}$.
\en
\end{Assumption}

Let us briefly discuss the importance of the above assumptions. The monotonicity  condition \eqref{eq:mono_d} in  (H.\ref{assum:coeff_d}(\ref{assum:mono_d})) is weaker than the finite-dimensional analogue of  the strong monotonicity condition \eqref{eq:coercive}, in the sense that \eqref{eq:mono_d} involves only the standard Euclidean norm instead of discrete Sobolev norms. The monotonicity, along with the Lipschitz continuity in  (H.\ref{assum:coeff_d}(\ref{assum:regu_d})), ensures the well-posedness of \eqref{eq:sde_thm_d} (see e.g.~\cite{mao1997}), and  allows us to   derive   precise  regularity estimates (in $L^{p}$-norms for $p\in [2,2+\eta)$) of the solution $Y^{x,d}$ to the SDE \eqref{eq:sde_thm_d} with respect to the coefficients and the initial condition. 

It is worth emphasizing that 
 (H.\ref{assum:coeff_d}) allows the operator norm of $A_d$ and the Lipschitz constants of the 
 nonlinear functions
$\mu_d$ and $\sigma_d$ to grow with respect to the dimension $d$,
 which 
 is crucial for applications to
  stiff SDEs  arising from Galerkin approximations of (controlled) SPDEs.
In fact, 
most existing results 
on overcoming the curse of dimensionality with DNNs 
(see e.g.~\cite{elbrachter2018,grohs2018,hutzenthaler2018,jentzen2018,hutzenthaler2019gradient,hutzenthaler2019})
are for  value functions associated with 
high-dimensional SDEs 
whose 
 diffusion coefficients
are affine with respect to the state variable
and 
both drift and diffusion coefficients are Lipschitz continuous uniformly with respect to the dimensions.
Note that it is easy to check that if $\mu_d,\sigma_d$ satisfy (H.\ref{assum:coeff_d}(\ref{assum:regu_d})) with a Lipschitz constant independent of the dimension $d$, then the coefficients satisfy (H.\ref{assum:coeff_d}(\ref{assum:mono_d})). In particular, our setting includes the representation result in \cite{jentzen2018} as a special case.

\color{black}

 We remark that both the monotonicity condition \eqref{eq:mono_d} and the Lipschitz continuity of $\mu_d$ are crucial for constructing networks with polynomial complexity to approximate the desired value functions. With the help of the monotonicity condition  (H.\ref{assum:coeff_d}(\ref{assum:mono_d})), we can demonstrate that both the regularity of the solution $Y^{x,d}$ to  \eqref{eq:sde_thm_d} and the error estimates of a corresponding partial-implicit Euler scheme  depend \textit{polynomially} on $\|A_d\|_{\rm{op}}$, $ [\mu_d]_1$ and $[\sigma_d]_1$, i.e., the Lipschitz constants of the coefficients (see Section \ref{sec:im_euler} for details; see also \cite{jentzen2018} for SDEs with merely Lipschitz continuous coefficients, for which the corresponding estimates depend exponentially on the Lipschitz constants of the coefficients).
These polynomial dependence results subsequently enable us to  construct  DNNs with polynomial complexities to approximate the  value functions induced by stiff SDEs, including those arising from Galerkin approximations of SPDEs.

On the other hand, the Lipschitz continuity of $\mu_d$ allows us to construct the desired DNNs through a linear-implicit Euler scheme of \eqref{eq:sde_thm_d}, which is implicit  in the linear part of the drift and remains explicit for the nonlinear part of the drift. 
In fact, to the best of our knowledge, if the function $\mu_d$ is not globally Lipschitz continuous,
then one needs to adopt a fully-implicit scheme, a tamed explicit scheme 
or an adaptive Euler scheme 
to obtain a convergent approximation of \eqref{eq:sde_thm_d}
in the $L^2$-norm.
These schemes in general  involve of
nonlinear mappings that are 
difficult to represent by ReLU networks;
in particular, 
the  fully-implicit scheme involves of the inverse of the  mapping
$y\mapsto y+\Delta t (A_dy- \mu_d(t,y))$
(see e.g.~\cite{mao2013}), 
the tamed explicit scheme involves of
the  mapping 
$y\mapsto\frac{\mu_d(t,y)}{1+\Delta t \|\mu_d(t,y)\|}$ (see e.g.~\cite{hutzenthaler2012}),
while the adaptive Euler scheme 
involves a non-uniform random stepsize which 
varies for different realisations of the Brownian motion
and needs to be constructed in a problem-dependent way (see e.g.~\cite{fang2016}).
\color{black}

Finally,  instead of  approximating directly $f_d$ on $\R^d$, (H.\ref{assum:coeff_d}(\ref{assum:terminal_d})) allows us to focus on approximating  $f_d$  on a hypercube, and then extend the approximation linearly outside the domain.
This is motivated by the fact that approximating a function by neural networks on a prescribed compact set  has been better understood than approximating the function globally on $\R^d$ (see e.g.~\cite{montanelli2017,petersen2017,yarotsky2017,e2018,bolcskei2019,grohs2019,guhring2019,opschoor2019}).
In particular,  since $f_d$ can admit   quadratic growth at infinity and ReLU networks can only generate piecewise linear functions, for a given small enough  $\eps$, there exists no ReLU network $\phi$ such that the inequality $|f_d(x)-[\cR_\varrho(\phi)](x)|\le \eps (1+ \|x\|^2)$  holds for  all $x\in \R^d$.
Therefore, we  adopt a two-step approximation by first  approximating $f_d$ with a suitable Lipschitz continuous function $f_{d,D}$, and then representing $f_{d,D}$ by a ReLU network on $\R^d$ with a desired accuracy;  see Proposition \ref{prop:quadratic} for the representation results for weighted square functions, which are the commonly used cost functions for PDE-constrained optimal control problems.

To construct neural networks with the desired complexities, 
we shall assume that the family of  functions $(\mu_d,\sigma_d)_{d\in \N}$ and  $(f_{d,D})_{d\in \N,D>0}$ can be approximated  by ReLU networks without curse of dimensionality.
\begin{Assumption}\l{assum:perturb_d}
Assume the notation of  (H.\ref{assum:coeff_d}).  Let $\kappa_1\ge 0$ and  $\varrho:\R\to \R$ be the function satisfying  $\varrho(x)=\max(0,x)$ for all $x\in \R$. Let
 $(\phi^\mu_{\eps,d},\phi^{\sigma,i}_{\eps,d},\phi^{f}_{\eps,d,D})_{\eps,d,D, i}\subset \cN$, $\eps\in (0,1], d\in \N, D>0,  i=1,\ldots, d$, be a family of DNNs with the following properties, for any given $d\in \N$, $\eps\in (0,1]$ and $D>0$:

\bn[(a)]
\item  \l{assum:perturb_d_archi}
The DNNs $(\phi^\mu_{\eps,d},\phi^{\sigma,i}_{\eps,d})_{i}$ have the same architecture, i.e.,
$$
\phi^\mu_{\eps,d},\;\phi^{\sigma,i}_{\eps,d}\in \cN^{d+1,N^{\eps,d}_1,\ldots, N^{\eps,d}_{L_{\eps,d}-1},d}_{L_{\eps,d}}, \q  i=1,\ldots, d,
$$
for some integers $L_{\eps,d},N^{\eps,d}_1,\ldots, N^{\eps,d}_{L_{\eps,d}-1}\in \N$, depending on $d$ and $\eps$.
\item  \l{assum:perturb_d_poly}
The DNNs $(\phi^\mu_{\eps,d},\phi^{\sigma,i}_{\eps,d},\phi^{f}_{\eps,d,D})$ admit the  following complexity estimates:
$$
\sC(\phi^\mu_{\eps,d})+\sum_{i=1}^d\sC(\phi^{\sigma,i}_{\eps,d})\le  \kappa_1 d^{\kappa_1}\eps^{-\kappa_1}, \q 
\sC(\phi^{f}_{\eps,d,D})\le \kappa_1 d^{\kappa_1}D^{\kappa_1}\eps^{-\kappa_1}.
$$
\item \l{assum:perturb_d_approx}
The realizations  
${\mu}^\eps_d=\cR_\varrho(\phi^\mu_{\eps,d})$, ${\sigma}^\eps_d=(\cR_\varrho(\phi^{\sigma,1}_{\eps,d}),\ldots,\cR_\varrho(\phi^{\sigma,d}_{\eps,d}) )$, and  ${f}^\eps_{d,D}=\cR_\varrho(\phi^{f}_{\eps,d,D})$ admit the following approximation properties: for all $t\in [0,T]$ and $x\in \R^d$,
$$
\|{\mu}_d(t,x)-{\mu}^\eps_d(t,x)\|+\|{\sigma}_d(t,x)-{\sigma}^\eps_d(t,x)\|\le \eps\kappa_1d^{\kappa_1}, \q |{f}_{d,D}(x)-{f}^\eps_{d,D}(x)|\le \eps\kappa_1 d^{\kappa_1}D^{\kappa_1}.
$$
\en
\end{Assumption}

Since a ReLU network can be extended to an arbitrary depth and width without changing its realization (Lemma \ref{lemma:extension}), we assume without loss of generality in (H.\ref{assum:perturb_d}(\ref{assum:perturb_d_archi})) that  $(\phi^\mu_{\eps,d},\phi^{\sigma,i}_{\eps,d})_{ i=1,\ldots, d}$ have the same  architecture to simplify our analysis.

The conditions (H.\ref{assum:perturb_d}(\ref{assum:perturb_d_poly}),(\ref{assum:perturb_d_approx})) imply 
 the function $(\mu_d,\sigma_d)_{d\in \N}$ and  $(f_{d,D})_{d\in \N,D>0}$ can be approximated  by ReLU networks with
 polynomial complexity in $\eps$, $d$ and $D$. These conditions clearly hold  for most sensible discretizations  of linear SPDEs, such as the  Zakai equation (see e.g.~\cite{ito1996,giles2012}):
$$
dy(t)+Ay(t)\,dt=Gy(t)\,dB(t), \, t\in (0,T], \q y(0)=y_0,
$$ 
where $A$ and $G$ are  second-order and first-order linear differential operators, respectively. Moreover,  by virtue of the fact that  ReLU networks can efficiently represent the pointwise maximum/minimum operations (see Proposition \ref{prop:min_max}), one  can see  (H.\ref{assum:perturb_d}(\ref{assum:perturb_d_poly}),(\ref{assum:perturb_d_approx})) also
hold for the discretizations of the  following  Hamilton--Jacobi--Bellman--Isaacs equation, since the (discretized) Hamiltonian can be  \textit{exactly} expressed  by ReLU networks:
$$
dy(t,x)+\big(-\nu (\Delta y)(t,x)-H(t,x,y(t,x),(\nabla_xy)(t,x))\big)dt=0, \q (t,x)\in (0,T]\t \cD,
$$
where $\nu>0$, $\cD$ is   a bounded open set in $\R^m$,  and the Hamiltonian $H:[0,T]\t \cD\t \R\t \R^m\to\R$ is given by:
$$
H(t,x,u,p)=\inf_{\a\in \bA}\sup_{\b\in \bB} \big[b(t,x,\a,\b)^Tp+c(t,x,\a,\b)u+\ell(t,x,\a,\b)\big],
$$
and $\bA,\bB$ are two given finite sets. Finally, for general semilinear PDEs with bounded solutions, one may consider an equivalent semilinear PDE by truncating the nonlinearity outside a compact set, and approximate the truncated coefficients by  DNNs.


Finally, we remark  that 
(H.\ref{assum:coeff_d}(\ref{assum:terminal_d})) and (H.\ref{assum:perturb_d}(\ref{assum:perturb_d_approx})) essentially assume that 
for any given $D>0$, 
there exists a deep ReLU network approximating the terminal function $f_{d}|_{B_\infty(D)}$ with polynomial complexity,
and the difference between the terminal function $f_d$ and the deep ReLU network can be controlled by the quadratic growth of $f_d$ outside the hypercube $B_\infty(D)$.
We  refer  the reader to Proposition   \ref{prop:quadratic}, where we verify (H.\ref{assum:coeff_d}(\ref{assum:terminal_d})) and (H.\ref{assum:perturb_d}(\ref{assum:perturb_d_approx})) for a class of quadratic cost functions.


Now we are ready to state one of  the main results of this paper, which shows that one can construct DNNs with polynomial complexity to approximate the value functions induced by nonlinear stiff SDEs.
Similar representation results have been shown in \cite{grohs2018} for SDEs with 
affine drift and diffusion coefficients, and in \cite{jentzen2018} for SDEs with nonlinear drift and constant diffusion coefficients. Our results extend these results to SDEs with time-inhomogeneous nonlinear drift and  diffusion coefficients. Moreover, we allow the Lipschitz constants of the coefficients to grow with the dimension $d$, which is crucial for the application to SPDE-constrained optimal control problems.  The proof of this theorem is given in Section  \ref{sec:proof_expression_sde}.

\begin{Theorem}\l{thm:expression_sde}
Suppose (H.\ref{assum:coeff_d}) and (H.\ref{assum:perturb_d}) hold.
For each $d\in \N$,  let $v_d:\R^d\to \R$ be the  function 
satisfying for all $x\in \R^d$  that
$v_d(x)= \ex[f_d(Y^{x,d}_T)]$ with  $(Y^{x,d}_t)_{t\in [0,T]}$ being the  solution to  \eqref{eq:sde_thm_d}, 
and let   $\nu_d$ be a probability measure on $\R^d$ satisfying $\int_{\R^d}\|x\|^{4+\eta}\,\nu_d(dx)\le \tau d^\tau$,
with the same constant $\eta$ as in (H.\ref{assum:coeff_d}),  and some constant $\tau>0$ independent of $d$.

Then there exists a family of DNNs $(\psi_{\eps,d})_{\eps\in (0,1],d\in \N}$ and a constant $c>0 $, 
depending only on   $\b,\eta,\kappa_1,\kappa_2,\tau$ and $T$,  such that for all $d\in \N$, $\eps\in (0,1]$, we have $\sC(\psi_{\eps,d})\le cd^c \eps^{-c}$, $\cR_\varrho(\psi^\eps_d)\in C(\R^d; \R)$ and  
$$
\bigg(\int_{\R^d} |v_d(x)-[\cR_\varrho(\psi_{\eps,d})](x)|^2\, \nu_d(dx)\bigg)^{1/2}<\eps.
$$
\end{Theorem}

The following result is a direct consequence of Theorem \ref{thm:expression_sde}
and the Feynman-Kac formula in \cite[Theorem 2.2]{pardoux1998}, 
which shows one can approximate the viscosity solution to a Kolmogorov backward PDE with stiff coefficients on a bounded domain without curse of dimensionality.
The proof will be postponed to Section \ref{sec:proof_expression_sde}.

\begin{Corollary}\l{cor:pde}
Suppose (H.\ref{assum:coeff_d}) and (H.\ref{assum:perturb_d}) hold.
For each $d\in \N$,  let  $u_d$ be the unique continuous viscosity solution to the following PDE with at most quadratic growth at infinity: 
\bb\l{eq:linear_pde}
\tfrac{\p u_d}{\p t}(t,x)+\tfrac{1}{2}\textnormal{tr}\big(\sigma_d(t,x)\sigma^T_d(t,x)(\textnormal{Hess}_xu_d)(t,x)\big)
+(-A_dx+\mu_d(t,x))^T(\nabla_xu_d)(t,x)=0
\ee
for all $(t,x)\in [0,T)\t \R^d$, and $u_d(T,x)=f_d(x)$ for $x\in \R^d$.
Then there exists a family of DNNs $(\psi_{\eps,d})_{\eps\in (0,1],d\in \N}$ and a constant $c>0 $, 
depending only on   $\b,\eta,\kappa_1,\kappa_2$ and $T$,  such that for all $d\in \N$, $\eps\in (0,1]$, we have $\sC(\psi_{\eps,d})\le cd^c \eps^{-c}$, $\cR_\varrho(\psi^\eps_d)\in C(\R^d; \R)$ and  
$$
\bigg(\int_{[0,1]^d} |u_d(0,x)-[\cR_\varrho(\psi_{\eps,d})](x)|^2\, dx\bigg)^{1/2}<\eps.
$$

\end{Corollary}

\subsection{Expression rate for controlled SDEs with stiff coefficients}\l{sec:sde_control}
In this section, we extend the expression rates  in Section \ref{sec:sde_stiff}, and construct DNNs with polynomial complexity to approximate value functions associated with a sequence of controlled SDEs with stiff coefficients.

We start by introducing the set of admissible strategies.  Let  $M\in \N$ and  $\cT_M$ be the set of intervention times defined as:
\bb\l{eq:intervention_time}
\cT_M=\{\bar{t}_k\in [0,T]\mid \bar{t}_k=kT/M, \, k=0,\ldots, M\}.
\ee
For each $d\in \N$, we consider the following piecewise-constant, deterministic strategies: for $i=1,2$,
\begin{align}\l{eq:U_12}
u_i\in \cU_{i,d}\coloneqq \{u_i:[0,T]\to U_{i,d}\mid u_i(t)=u_i(\bar{t}_k)\in U_{i,d}, \; \fa t\in [\bar{t}_k,\bar{t}_{k+1}), \, k=0,\ldots, M-1\},
\end{align}
where $U_{i,d}$, $i=1,2$, are  given nonempty finite subsets of $\R^{m_d}$ for some $m_d\in \N$. Note that  $u_i$ can be the coefficients of a parameterized control policy in the sense that if $u_i(t)=(u_{i,j}(\bar{t}_k))_{j=1}^{m_d}$ on $[t_k,t_{k+1})$, the state equation is controlled by a policy $\tilde{u}_i(t,x)=\sum_{j=1}^{m_d}u_{i,j}(\bar{t}_k) e_j(t,x)$ on $[t_k,t_{k+1})$, where $\{e_j(t,x)\}_{j=1}^{m_d}$ are some prescribed basis functions.

Now for each $d\in \N$, we consider  a two-player zero-sum stochastic differential game, where the ``inf-player'' aims to minimize a particular function over all strategies $u_1\in\cU_{1,d}$, while the ``sup-player'' aims to maximize it over all strategies $u_2\in\cU_{2,d}$. The value function  $v_d:\R^d\to \R$ is given by:
\bb\l{eq:value_control}
v_d(x)\coloneqq \inf_{u_1\in \cU_{1,d}}\sup_{u_2\in \cU_{2,d}}\ex\bigg[f_d(Y^{x,d,u_1,u_2}_T)+g_{d}(u_1,u_2)\bigg], \q x\in \R^d,
\ee
where   for each $x\in \R^d$, $u_i\in \cU_{i,d}$ (see \eqref{eq:U_12}), $i=1,2$, $Y^{x,d,u_1,u_2}=(Y^{x,d,u_1,u_2}_t)_{t\in [0,T]}$ is the strong solution to the following $d$-dimensional controlled SDE:
 \bb\l{eq:sde_thm_d_ctrl}
 dY_t=(-A_dY_t+\mu_d(t,Y_{t},u_1,u_2))\,dt+ {\sigma}_d(t,Y_{t},u_1,u_2)\,dB_t, \q t\in (0,T];\q Y_0=x,
\ee
 with a $d$-dimensional Brownian motion  $(B_t)_{t\in [0,T]}$ defined on a probability space $(\Om, \cF , \bP )$. 
For simplicity, here we do not take into account any running costs of the state process $Y^{x,d,u_1,u_2}$, but 
it is straightforward to extend our results to  control problems with running costs. Moreover, the result would not change if the linear part of the drift is also controlled.

Note that, if the coefficients $\mu_d$ and $\sigma_d$ of \eqref{eq:sde_thm_d_ctrl} are independent of the control parameters and the cost function 
$g_d:\cU_{1,d}\t \cU_{1,d}\to \R$ in \eqref{eq:value_control} is  the zero function,
then the value function 
defined in \eqref{eq:value_control}
reduces to the function 
defined in \eqref{eq:value_unctrl}.
\color{black}

We then state the  assumptions on the coefficients of \eqref{eq:sde_thm_d_ctrl} for deriving the expression rates of DNNs. Roughly speaking, we assume (H.\ref{assum:coeff_d}) and (H.\ref{assum:perturb_d}) hold uniformly in terms of the control parameters. However, we would like to point out that even though the functions $\mu_d,\sigma_d$ are continuous in time, the controlled drift and diffusion of \eqref{eq:sde_thm_d_ctrl} are discontinuous in time due to the jumps in the control processes.

\begin{Assumption}\l{assum:coeff_d_ctrl}
Let $\b,\kappa_0,T\ge 0$, $\eta>0$ and $M\in \N$ be fixed constants. Let the set $\cT_M$ be defined as in \eqref{eq:intervention_time}. For all $d\in \N$ and $D>0$, let $A_d\in \R^{d\t d}$,
$U_d=U_{1,d}\t U_{2,d}$ be a \textcolor{blue}{nonempty} subset of  $\R^{m_d}$ for some $m_d\in \N$, \  and ${\mu}_d:[0,T]\t \R^d\t U_d\to \R^d$, ${\sigma}_d:[0,T]\t \R^d \t U_d\to \R^{d\t d}$, $f_d, f_{d,D}:\R^d\to \R$, $g_{d}:U_{d}\to \R$  be measurable functions with the following properties:
 \bn[(a)]
\item\l{assum:regu_d_ctrl}
For all $u\in U_{d}$, the matrix $A_d$ and the functions $\mu_d(\cdot,\cdot,u):[0,T]\t \R^d\to \R$, $\sigma_d(\cdot,\cdot,u):[0,T]\t \R^d\to \R^{d\t d}$ satisfy (H.\ref{assum:coeff_d}(\ref{assum:mono_d}),(\ref{assum:A_d}),(\ref{assum:regu_d})) with the constants 
$\b,\kappa_0,\eta$. 
\item The functions $f_d$ and $f_{d,D}$ satisfy (H.\ref{assum:coeff_d}(\ref{assum:terminal_d})).
\item\l{assum:d_ctrl_set}
The cardinality of the set $U_d$ satisfies  $|U_{d}|\le \kappa_0d^{\kappa_0}$.
\en
\end{Assumption}

\begin{Assumption}\l{assum:perturb_d_ctrl}
Assume the notation of (H.\ref{assum:coeff_d_ctrl}). Let $\kappa_1\ge 0$ and  $\varrho:\R\to \R$ be the function satisfying  $\varrho(x)=\max(0,x)$ for all $x\in \R$. Let
 $(\phi^\mu_{\eps,d},\phi^{\sigma,i}_{\eps,d},\phi^{f}_{\eps,d,D})_{\eps,d,D, i}\subset \cN$, $\eps\in (0,1], d\in \N, D>0,  i=1,\ldots, d$, be a family of DNNs with the following properties, for any given $d\in \N$, $\eps\in (0,1]$ and $D>0$:

\bn[(a)]
\item 
The DNNs $(\phi^\mu_{\eps,d},\phi^{\sigma,i}_{\eps,d})_{i}$ have the same architecture with the input dimension $d+m_d+1$.
\item  
The complexities of the DNNs $(\phi^\mu_{\eps,d},\phi^{\sigma,i}_{\eps,d},\phi^{f}_{\eps,d,D})$ satisfy (H.\ref{assum:perturb_d}(\ref{assum:perturb_d_poly})) with the constant $\kappa_1$.

\item 
The realizations  
${\mu}^\eps_d=\cR_\varrho(\phi^\mu_{\eps,d})$, ${\sigma}^\eps_d=(\cR_\varrho(\phi^{\sigma,1}_{\eps,d}),\ldots,\cR_\varrho(\phi^{\sigma,d}_{\eps,d}) )$, and  ${f}^\eps_{d,D}=\cR_\varrho(\phi^{f}_{\eps,d,D})$ admit the following approximation properties: for all $t\in [0,T]$, $x\in \R^d$ and $u\in U_{d}$,
$$
\|{\mu}_d(t,x,u)-{\mu}^\eps_d(t,x,u)\|+\|{\sigma}_d(t,x,u)-{\sigma}^\eps_d(t,x,u)\|\le \eps\kappa_1d^{\kappa_1}, \q |{f}_{d,D}(x)-{f}^\eps_{d,D}(x)|\le \eps\kappa_1 d^{\kappa_1}D^{\kappa_1}.
$$
\en
\end{Assumption}

We remark that (H.\ref{assum:coeff_d_ctrl}(\ref{assum:regu_d_ctrl}),(\ref{assum:d_ctrl_set})) are natural assumptions if the 
controlled stiff SDEs \eqref{eq:sde_thm_d_ctrl} arise from a discretization of controlled SPDEs. 
For example, in the deterministic setting, one can consider the following optimal control problems 
within a spatial domain $\Om\subset\R^p$: 
\begin{align}\l{eq:parabolic_ctrl}
\cV(y_0)=\min_{u\in \cU}\cJ(u),\q \cJ(u)\coloneqq\|y(T)-\bar{y}\|^2_{L^2(\Om)}&+ \|u\|^2_{\mathbb{\R}^\mathfrak{m}},
\end{align}
where 
 the set of admissible controls $\cU$ is a compact subset of $\R^\mathfrak{m}$, $\bar{y}\in L^2(\Om)$ is the desired terminal state,  and $y$ is governed by a controlled semilinear parabolic PDE with initial condition $y_0\in L^2(\Om)$:
$$
\textstyle
 \tfrac{\p}{\p t} y -\Delta y+ G(y)=\sum_{i=1}^\mathfrak{m} u_i e_i(t,x)
\; \textnormal{in $\Om\t (0,T)$},
\q y=0 \; \textnormal{on $\p\Om\t (0,T)$},
\q y=y_0 \; \textnormal{in $\Om\t \{t=0\}$}
$$
with given Lipschitz function $G$ and  sufficiently regular basis functions $\{e_i\}_{i=1}^\mathfrak{m}$.
Note that finitely many control parameters appear frequently in practical  applications of optimal control theory, since it is  difficult to implement control strategies that  vary arbitrarily in time and space; see e.g.~\cite{merino2010_1,merino2010_2} for elliptic optimal control problems with finite dimensional control spaces.
Then it is clear that
the  discrete version of \eqref{eq:parabolic_ctrl}
(in both the space and control variables)
is a special case of the zero-sum game \eqref{eq:value_control}
whose coefficients satisfy (H.\ref{assum:coeff_d_ctrl}) (with $M=1$ in \eqref{eq:U_12}).
\color{black}

Now we are ready to present the main theorem in this section, 
which shows one can represent the value function \eqref{eq:value_control}  by DNNs without curse of dimensionality, whose  proof will be deferred to  Section  \ref{sec:proof_expression_sde_ctrl}.

Note that here the value function \eqref{eq:value_control}
is induced by optimizing the cost functional over  deterministic control strategies (i.e., open-loop controls).
For general stochastic games with adapted stochastic strategies (i.e., closed-loop controls),
the value function can be identified as the solution of a $d$-dimensional fully-nonlinear HJBI equation (see e.g.~\cite{touzi2012}),
for which the analysis of DNN approximation rates is more involved 
(see \cite{hutzenthaler2018,hutzenthaler2019gradient,hutzenthaler2019} for 
some results on overcoming the curse of dimensionality with DNNs for some semilinear PDEs).

\color{black}
\begin{Theorem}\l{thm:expression_sde_ctrl}
Suppose (H.\ref{assum:coeff_d_ctrl}) and (H.\ref{assum:perturb_d_ctrl}) hold.
For each $d\in \N$,  let $v_d:\R^d\to \R$ be the value function defined  in \eqref{eq:value_control}, and let 
 $\nu_d$ be a probability measure on $\R^d$ satisfying $\int_{\R^d}\|x\|^{4+\eta}\,\nu_d(dx)\le \tau d^\tau$,
 with the same constant $\eta$ as in (H.\ref{assum:coeff_d}),  and some constant $\tau>0$ independent of $d$.

Then there exists a family of DNNs $(\psi_{\eps,d})_{\eps\in (0,1],d\in \N}$ and a constant $c>0 $, 
depending only on   $\b,\eta,\kappa_1,\kappa_2,\tau, M$ and $T$,  such that for all $d\in \N$, $\eps\in (0,1]$, we have $\sC(\psi_{\eps,d})\le cd^c \eps^{-c}$, $\cR_\varrho(\psi^\eps_d)\in C(\R^d; \R)$ and  
$$
\bigg(\int_{\R^d} |v_d(x)-[\cR_\varrho(\psi_{\eps,d})](x)|^2\, \nu_d(dx)\bigg)^{1/2}<\eps.
$$
\end{Theorem}

\section{ReLU network calculus}\l{sec:cal}
In this section, we shall discuss several basic operations   to construct new DNNs from  existing ones.  
We shall also establish some fundamental  results on  the representation flexibility of DNNs by following the setting  of Definition \ref{def:DNN}, which are essential for our subsequent analysis.

Recall that it has been shown in \cite{elbrachter2018,grohs2018} that linear combination and composition of a finite number of ReLU DNNs can  be realized by a ReLU DNN with polynomial complexity. Moreover,  the identity function can  be implemented as a ReLU network with one hidden layer.  The precise  statements of these results will be given in Appendix \ref{appendix} for completeness.

The following proposition  presents a  \textit{global} approximation result for the weighted square function on $\R^d$.   
Since the quadratic growth of the weighted square function prevents us to directly approximate it by a ReLU neural network, we shall employ a two-step  approximation by first approximating the  function  on a prescribed compact set and then linearly extending it  outside the bounded domain.  

\begin{Proposition}\l{prop:quadratic}
Let  $\varrho:\R\to \R$ be the activation function defined as $\varrho(x)=\max(0,x)$ for all $x\in \R$.
Let $d\in \N$,  $\b=(\b_m)_{m=1}^d\in \R$, and $f:\R^d\to \R$ be the weighted  square function defined by $f(x)\coloneqq \sum_{m=1}^d \b_m x_m^2$ for all $x=(x_1,\ldots, x_d)^T\in \R^d$. Then
  for any $D>0,\eps\in (0,1/2)$, there exists a function $f_{d,D}:\R^d\to \R$ and a DNN $\phi_{\eps,d,D}$ satisfying
\begin{alignat*}{2}
|f(x)-f_{d,D}(x)|&\le \|\b\|_\infty\|x\|^{2}1_{B_\infty(D)^c}(x),&\q \;  |f_{d,D}(x)-f_{d,D}(y)|&\le 2\|\b\|_\infty d^{1/2}D\|x-y\|,\\
 |f_{d,D}(x)-[\cR_\varrho(\phi_{d,\eps,D})](x)|&\le \|\b\|_\infty dD^2\eps,   &
 \sC(\phi_{d,\eps,D}))&\le Cd^2\log (\eps^{-1})+d+1,
\end{alignat*}
where $B_\infty(D)\coloneqq\{x\in \R^d\mid |x_m|<D, \fa m\}$, $\|\b\|_\infty=\sup_{m}|\b_m|$, and $C$ is a constant independent of $\b,D,d$ and $\eps$. 
\end{Proposition}
\begin{proof}
We start by  constructing a \textit{global} approximation of  the (one-dimensional) squaring operation $x\mapsto x^2$. Since this construction is similar to that  in \cite[Proposition III.2]{grohs2019}, we only repeat the main steps for reader's convenience. Recall that it has been shown in \cite[Proposition III.1]{grohs2019} (see also \cite{yarotsky2017}) that for any $\eps\in (0,1/2)$, one can use  the ``tooth'' (or ``mirror'') function $g:\R\to [0,1]$:
$$
g(x)=\begin{cases} 2x, & 0\le x<1/2;\\
2(1-x), & 1/2\le x<1;\\
0, & \textnormal{otherwise}, \end{cases}
$$
and  its $s$-fold composition (i.e., the ``sawtooth'' function) to construct a DNN $\phi_{1,\eps}$ satisfying $\cL(\phi_{1,\eps}) \le C \log (\eps^{-1})$, $\sC(\phi_{1,\eps})\le C\log (\eps^{-1})$, $[\cR_\varrho(\phi_{1,\eps})](0)=0$, $[\cR_\varrho(\phi_{1,\eps})](x)=x$ for $x\not\in[0,1]$, and 
$|[\cR_\varrho(\phi_{1,\eps})](x) - x^2|_{L^\infty[0,1]}\le  \eps$ for some constant $C$ independent of $\eps$.
Then for any given  $D>0$, we shall consider the function  $f_{\eps,1,D}: x\mapsto D^2[\cR_\varrho(\phi_{1,\eps})](|x|/D)$, $x\in \R$,
and approximate the square function by the following function: 
\bb\l{eq:f_i}
f_{1,D}(x)=\begin{cases} x^2, & |x|\le   D ;\\
 D|x|, & \textnormal{otherwise}. \end{cases}
\ee
It follows directly from the properties of $\phi_{1,\eps}$ that $f_{\eps,1,D}(x)=f_{1,D}(x)$ for $|x|>D$ and $|f_{\eps,1,D}(x)- f_{1,D}(x)|_{L^\infty[-D,D]}\le  D^2\eps$. 
Since $f_{\eps,1,D}$ is  a composition of the functions $x\mapsto [\cR_\varrho(\phi_{1,\eps})](x)$ and $x\to |x|=\varrho(x)+\varrho(-x)$, 
we know it is the realization of a ReLU network $\phi_{\eps,1,D}$ with complexity $\sC(\phi_{1,\eps,D})\le C\log (\eps^{-1})$, 
 for some constant $C$ independent of $\eps$ and $D$.

Now for any given $(\b_m)_{m=1}^d\in \R$,  we consider the   functions
$f_{d,D}(x)=\sum_{m=1}^d\b_m f_{1,D}(x_m)$ and $f_{\eps, d,D}(x)=\sum_{m=1}^d\b_mf_{\eps,1,D}(x_m)$ for all $x=(x_1,\ldots, x_d)^T\in \R^d$. It is straightforward to verify that it holds for all  $x\in B_\infty(D)$ that $f_{d,D}(x)-f(x)=0$, and for all $x\in B_\infty(D)^c$ that 
$$
|f_{d,D}(x)-f(x)|\le \sum_{m=1}^d |\b_m| |f_{1,D}(x_m)- x_m^2|\le \sum_{m=1}^d |\b_m|  (x_m^2-D|x_m|)1_{\{|x_m|>D\}}\le 
\|\b\|_\infty\|x\|^2,
$$
where we denote $\|\b\|_\infty=\sup_{m}|\b_m|$. Also $f_{d,D}$ is Lipschitz continuous, i.e., for all $x,y\in \R^d$,
\begin{align*}
|f_{d,D}(x)-f_{d,D}(y)|&\le \sum_{m=1}^d|\b_m||f_{1,D}(x_m)-f_{1,D}(y_m)|\\
&\le 2D\|\b\|_\infty\sum_{m=1}^d|x_m-y_m|\le 2D\|\b\|_\infty d^{1/2}\|x-y\|.
\end{align*}
 Moreover, the following approximation property holds:
\begin{align*}
|f_{d,D}(x)-f_{\eps, d,D}(x)|&\le \sum_{m=1}^d|\b_m||f_{1,D}(x_m)-f_{\eps,1,D}(x_m)|\\
&= \sum_{m=1}^d|\b_m||f_{1,D}(x_m)-f_{\eps,1,D}(x_m)|1_{\{|x_m|\le D\}}\le \|\b\|_\infty dD^2\eps.
\end{align*}
Therefore, it remains to show $f_{\eps, d,D}$ is the realization of a ReLU network and estimate its complexity. The main tool to construct the desired ReLU network is a ``parallelization" of the network $\phi_{1,\eps,D}$ (see \cite{elbrachter2018}). Suppose that the network  $\phi_{1,\eps,D}$ is given by $\phi_{1,\eps,D}=((W_1,b_1),(W_2,b_2),\ldots, (W_{L},b_{L}))$, and the dimension is $\dim(\phi_{1,\eps,D})=(N_0,N_1,\ldots, N_{L-1}, N_L)$. Then we consider the DNN 
$\phi_{d,\eps,D}=((W'_1,b'_1),(W'_2,b'_2),\ldots, (W'_{L+1},b'_{L+1}))$, where we have for all $i=1,\ldots, L$,
\begin{align*}
W_i&=\begin{pmatrix} W_i &  &\\
& W_i & &\\ 
&  & \ddots &\\ 
&& & W_i \end{pmatrix}\in \R^{(dN_i)\t (dN_{i-1})},\q
b'_i=\begin{pmatrix} b_i\\
b_i \\ \vdots\\
b_i  \end{pmatrix},
\end{align*}
and $W'_{L+1}=\begin{pmatrix} \b_1, \cdots,\b_d \end{pmatrix}\in \R^{1\t dN_L}$ and  $b'_{L+1}=0$.
Then it is easy to see that $f_{\eps, d,D}(x)=[\cR_\varrho(\phi_{d,\eps,D})](x)$ for all $x\in \R^d$, and the complexity of $\phi_{d,\eps,D}$ is given by
$$
\sC(\phi_{d,\eps,D})=\sum_{l=1}^L (dN_i)(dN_{i-1}+1)+(d+1)\le d^2\bigg(\sum_{l=1}^L N_i(N_{i-1}+1)\bigg)+d+1
\le Cd^2\log (\eps^{-1})+d+1,
$$ 
some constant $C$ independent of $\eps,d$ and $D$.
\end{proof}

The next proposition  presents an operation involving linear combination and compositions   of networks with different architectures, which extends  Proposition 5.2 in \cite{grohs2018} for two networks with the same input-output dimensions (i.e., $M=2$ and $d'=0$)  to multiple networks with different input-output dimensions (i.e., $M\ge 2$ and $d'\in \N$). Such an extension is essential for the subsequent analysis of controlled SDEs with time-inhomogeneous coefficients and a nonlinear diffusion term. 

\begin{Proposition}\l{prop:add+comp}
Let  $\varrho:\R\to \R$ be the activation function defined as $\varrho(x)=\max(0,x)$ for all $x\in \R$. Let $M,L,L'\in \N$, $d\in \N$, $d'\in \N\cup \{0\}$, $u\in \R^{d'}$, and 
$(\phi_{m})_{m=1}^M\in \cN$ be  DNNs  such that  
 \begin{alignat*}{2}
\cL(\phi_{1})&=L, \q &&\dim(\phi_1)=(d,N^{(1)}_1,N^{(1)}_2\ldots, N^{(1)}_{L-1}, d),\\
\cL(\phi_{m})&=L', \q &&\dim(\phi_m)=(d+d',N^{(m)}_1,N^{(m)}_2\ldots, N^{(m)}_{L'-1},d),\q m\in \{2,\ldots,M\},
\end{alignat*}
for some $N^{(1)}_l, N^{(m)}_{l'}\in \N$, $ l=1,\ldots,L-1,l'=1,\ldots, L'-1, m=2,\ldots,M$.

Then there exists a DNN $\psi\in \cN$ such that the depth $\cL(\psi)=L+L'-1$, the  dimension
\bb
\dim(\psi)=\begin{cases} (d,N^{(1)}_1,N^{(1)}_2\ldots, N^{(1)}_{L-1}, d) & L'=1,\\
(d,N^{(1)}_1,N^{(1)}_2,\ldots, N^{(1)}_{L-1}, 2d+\sum_{m=2}^M N^{(m)}_1, \ldots, 2d+\sum_{m=2}^M N^{(m)}_{L'-1}, d), & L'\ge 2,
\end{cases}
\ee
and satisfies the following identity:
 \bb\l{eq:add+comp}
[ \cR_\varrho(\psi)](x)= [ \cR_\varrho(\phi_1)](x)+\sum_{m=2}^M[ \cR_\varrho(\phi_m)]([ \cR_\varrho(\phi_1)](x),u),\q x\in \R^{d}.
 \ee
 Assume in addition for the case  $L'\ge 2$ that,  
$N^{(1)}_{L-1}\le 2d+\sum_{m=2}^M N^{(m)}_{L'-1}$, and there exists  $m_0\in \{2,\ldots, M\}$ such that we have for all $i\in \{2,\ldots,L'-1\}$, $m\in\{2,\ldots, M\}$ that $N^{(m)}_i\le N^{(m_0)}_i$. Then it holds that 
\begin{align}\l{eq:complexity_add+comp}
\sC(\psi)&\le \sC(\phi_{1})+(M-1)^2\bigg(\sup_{m\in \{2,\ldots, M\}}\sC(\phi_{m})+\sC(\phi^{\textnormal{Id}}_{d,2})\bigg)^3,
\end{align}
where $\phi^{\textnormal{Id}}_{d,2}$  is the two-layer  representation of the $d$-dimensional identity function defined as in \eqref{eq:identity}.
\end{Proposition}

\begin{proof}
We shall assume the networks $(\phi_m)_{m=1}^M$ are given as follows, and construct the desired network $\psi$ differently based on whether  $L'=1$ or $L'\ge 2$:
 \begin{align*}
\phi_{1}&=((W^{(1)}_1,b^{(1)}_1),(W^{(1)}_2,b^{(1)}_2),\ldots, (W^{(1)}_{L},b^{(1)}_{L}))\in \cN^{d,N^{(1)}_1,N^{(1)}_2\ldots, d}_{L},\\
\phi_{m}&=((W^{(m)}_1,b^{(m)}_1),(W^{(m)}_2,b^{(m)}_2),\ldots, (W^{(m)}_{L'},b^{(m)}_{L'}))\in \cN^{d+d',N^{(m)}_1,N^{(m)}_2\ldots, d}_{L'}, \q m\in \{2,\ldots,M\}.
\end{align*}

Suppose  $L'=1,$ we shall consider the  DNN 
$\psi=((W_1,b_1),(W_2,b_2),\ldots, (W_{L},b_{L}))$,    where $(W_i,b_i)=(W^{(1)}_i,b^{(1)}_i)$ for $i\in \{1,\ldots, L-1\}$ and
$$
W_L=W^{(1)}_{L}+\sum_{m=2}^M W^{(m)}_{L}\begin{pmatrix} W^{(1)}_{L}\\ \textbf{0}\end{pmatrix}\in \R^{d\t N^{(1)}_{L-1}}, \q 
b_L=b^{(1)}_{L}+\sum_{m=2}^M \bigg(W^{(m)}_{L}\begin{pmatrix} b^{(1)}_{L}\\ u\end{pmatrix}+b^{(m)}_L\bigg),
$$
with $\textbf{0}\in \R^{d'\t N^{(1)}_{L-1}}$. Then we have for all $x\in \R^{N^{(1)}_{L-1}}$ that 
$$W_Lx+b_L=W^{(1)}_{L}x+b^{(1)}_{L}+\sum_{m=2}^M\bigg( W^{(m)}_{L}\begin{pmatrix} W^{(1)}_{L}x+b^{(1)}_{L}\\ u\end{pmatrix}+b^{(m)}_{L}\bigg),$$
which implies \eqref{eq:add+comp}. Also it is clear that $\sC(\psi)=\sC(\phi_{1})$.

Now let $L'\ge 2$, 
$\phi^{\textnormal{Id}}_{d,2}=((W^{\textnormal{Id}}_1,0),(W^{\textnormal{Id}}_2,0))\in \cN^{d,2d, d}_{2}$  be the DNN representation of the $d$-dimensional identity function defined as in \eqref{eq:identity}. We shall construct the desired DNN as follows:
$$
\psi=((W_1,b_1),(W_2,b_2),\ldots, (W_{L+L'-1},b_{L+L'-1})),
$$  
where for  $i\in \{1,\ldots, L-1\}$, we have $(W_i,b_i)=(W^{(1)}_i,b^{(1)}_i)$; for $i=L$, we have
\begin{align*}
W_i&=\begin{pmatrix} W^{\textnormal{Id}}_1W^{(1)}_L\\
W^{(2)}_1\begin{pmatrix} W^{(1)}_L\\ \textbf{0}\end{pmatrix} \\ \vdots\\
W^{(M)}_1\begin{pmatrix} W^{(1)}_L\\ \textbf{0}\end{pmatrix} \end{pmatrix}\in \R^{(2d+\sum_{m=2}^M N^{(m)}_1)\t N^{(1)}_{L-1}},\q
b_i=\begin{pmatrix} W^{\textnormal{Id}}_1b^{(1)}_L\\
W^{(2)}_1\begin{pmatrix} b^{(1)}_L\\ u\end{pmatrix}+b^{(2)}_1 \\ \vdots\\
W^{(M)}_1\begin{pmatrix} b^{(1)}_L\\ u\end{pmatrix}+b^{(M)}_1 \end{pmatrix},
\end{align*}
with $\textbf{0}\in \R^{d'\t N^{(1)}_{L-1}}$;  for $i \in \{L+1,\ldots L+L'-2\}$, we have
\begin{align*}
W_i&=\begin{pmatrix} W^{\textnormal{Id}}_1W^{\textnormal{Id}}_2 & & &\\
& W^{(2)}_{i-L+1} & &\\ 
&  & \ddots &\\ 
&& & W^{(M)}_{i-L+1} \end{pmatrix}\in \R^{(2d+\sum_{m=2}^M N^{(m)}_{i-L+1})\t (2d+\sum_{m=2}^M N^{(m)}_{i-L})},\q
b_i=\begin{pmatrix} 0\\
b^{(2)}_{i-L+1} \\ \vdots\\
b^{(M)}_{i-L+1}  \end{pmatrix};
\end{align*}
and for $i=L+L'-1$ we have
$$
W_{i}=\begin{pmatrix}W^{\textnormal{Id}}_2 & W^{(2)}_{L'} &\dots & W^{(M)}_{L'}\end{pmatrix}\in \R^{d\t (2d+\sum_{m=2}^M N^{(m)}_{L'-1})},\q b_{i}=\sum_{m=2}^M b^{(M)}_{L'}. 
$$
Note that we have for all $x\in \R^{N^{(1)}_{L-1}}$ that 
$$
W_Lx+b_L
=
\begin{pmatrix} W^{\textnormal{Id}}_1(W^{(1)}_Lx+b^{(1)}_L)\\
W^{(2)}_1\begin{pmatrix} W^{(1)}_Lx+b^{(1)}_L\\ u\end{pmatrix}+b^{(2)}_1\\ \vdots\\ 
W^{(M)}_1\begin{pmatrix} W^{(1)}_Lx+b^{(1)}_L\\ u\end{pmatrix}+b^{(M)}_1 \end{pmatrix},
$$
 for all $i \in \{L+1,\ldots L+L'-2\}$, $x\in\R^{2d}$, $y^{(m)}\in \R^{N^{(m)}_{i-L}}$, $m\in \{2,\ldots,M\}$ that
$$
W_i\begin{pmatrix} x\\ y^{(2)}\\\vdots\\ y^{(M)}\end{pmatrix}+b_i=\begin{pmatrix} W^{\textnormal{Id}}_1(W^{\textnormal{Id}}_2x) \\
 W^{(2)}_{i-L+1}y^{(2)}+b^{(2)}_{i-L+1}  \\ \vdots\\
 W^{(M)}_{i-L+1}y^{(M)}+ b^{(M)}_{i-L+1}\end{pmatrix},
$$
and for all $i =L+L'-1$, $x\in \R^{2d}$, $y^{(m)}\in \R^{N^{(m)}_{i-L}}$, $m\in \{2,\ldots,M\}$ that
$$
W_i\begin{pmatrix} x\\ y^{(2)}\\\vdots\\ y^{(M)}\end{pmatrix}+b_i=W^{\textnormal{Id}}_2x+\sum_{m=2}^M
(W^{(m)}_{L'}y^{(m)}+b^{(m)}_{L'}).
$$
The above identities together with the fact that the DNN $\phi^{\textnormal{Id}}_{d,2}$ represents the $d$-dimensional identity function, i.e.,~$W^{\textnormal{Id}}_2\varrho^*(W^{\textnormal{Id}}_1x)=x$ for all $x\in \R^d$, enable us to conclude \eqref{eq:add+comp} by induction on $i$.

Now we turn to estimate the complexity of $\psi$, which is given by
\begin{align*}
\sC(\psi)=&\sum_{i=1}^{L-1}N^{(1)}_i(N^{(1)}_{i-1}+1)+(2d+\sum_{m=2}^M N^{(m)}_1)( N^{(1)}_{L-1}+1)\\
&+\sum_{i=L+1}^{L+L'-2}\bigg(2d+\sum_{m=2}^M N^{(m)}_{i-L+1}\bigg) \bigg(2d+\sum_{m=2}^M N^{(m)}_{i-L}+1\bigg)
+d(2d+\sum_{m=2}^M N^{(m)}_{L'-1}+1)\\
=&\sum_{i=1}^{L-1}N^{(1)}_i(N^{(1)}_{i-1}+1)+(\textbf{i}+\sum_{m=2}^M N^{(m)}_1)( N^{(1)}_{L-1}+1)\\
&+\sum_{i=1}^{L'-2}\bigg(\textbf{i}+\sum_{m=2}^M N^{(m)}_{i+1}\bigg) \bigg(\textbf{i}+\sum_{m=2}^M N^{(m)}_{i}+1\bigg)
+d(\textbf{i}+\sum_{m=2}^M N^{(m)}_{L'-1}+1),
\end{align*}
where we denote $\textbf{i}=2d$. Now using the assumptions that  $N^{(1)}_{L-1}\le \textbf{i}+\sum_{m=2}^M N^{(m)}_{L'-1}$, and $N^{(m)}_i\le N^{(m_0)}_i$ for all $i\in \{2,\ldots, L'-1\}$, $m\in\{2,\ldots, M\}$, we have
\begin{align*}
\sC(\psi)\le\, &\sum_{i=1}^{L-1}N^{(1)}_i(N^{(1)}_{i-1}+1)+(\textbf{i}+\sum_{m=2}^M N^{(m)}_1)( \textbf{i}+\sum_{m=2}^M N^{(m)}_{L'-1}+1)\\
&+\sum_{i=1}^{L'-2}\bigg(\textbf{i}+\sum_{m=2}^M N^{(m)}_{i+1}\bigg) \bigg(\textbf{i}+\sum_{m=2}^M N^{(m)}_{i}+1\bigg)
+d(\textbf{i}+\sum_{m=2}^M N^{(m)}_{L'-1}+1)\\
\le\, & \sC(\phi_{1})+(M-1)^2\bigg[(\textbf{i}+ N^{(m_0)}_1)( \textbf{i}+ N^{(m_0)}_{L'-1}+1)\\
&+\sum_{i=1}^{L'-2}\bigg(\textbf{i}+ N^{(m_0)}_{i+1}\bigg) \bigg(\textbf{i}+ N^{(m_0)}_{i}+1\bigg)
+N^{(m_0)}_{L'}(\textbf{i}+ N^{(m_0)}_{L'-1}+1)\bigg].
\end{align*}
Then from the same arguments as \cite[Proposition 5.3]{grohs2018} (c.f.~equation (124) in \cite{grohs2018}), we can bound the terms in the square bracket and deduce that  
\begin{align*}
\sC(\psi)&\le   \sC(\phi_{1})+(M-1)^2\big[\big(\sC(\phi_{m_0})+\sC(\phi^{\textnormal{Id}}_{d,2})\big)^3\big],
\end{align*}
which leads to the  complexity estimate \eqref{eq:complexity_add+comp} by using $\sC(\phi_{m_0})\le \sup_{m\in \{2,\ldots, M\}}\sC(\phi_{m})$.
\end{proof}

\begin{Remark}\l{rmk:add+comp}

Note that the complexity of the resulting network $\psi$ is additive to that of the  network $\phi_1$. Moreover,  for fixed networks $(\phi_m)_{m=2}^M$, if we start with a network $\phi_1$ whose the last hidden layer's  dimension  satisfies  $N^{(1)}_{L-1}\le 2d+\sum_{m=2}^M N^{(m)}_{L'-1}$, our construction ensures that the dimension of the last hidden layer of the resulting network $\psi$ also enjoys the same property. These two important observations enable us to iteratively apply Proposition \ref{prop:add+comp}, and construct a  network with desired complexity in Sections \ref{sec:proof_expression_sde} and \ref{sec:proof_expression_sde_ctrl}.

\end{Remark}

We end this section with the fact that taking pointwise maximum or minimum preserves the property of being represented by a ReLU DNN. One can find similar results in \cite[Lemma A.3]{arora2018}, where the authors adopt a different notation of neural network by allowing connections between  nodes in non-consecutive  layers.

\begin{Proposition}\l{prop:min_max}
Let  $\varrho:\R\to \R$ be the activation function defined as $\varrho(x)=\max(0,x)$ for all $x\in \R$. Let  $(\phi_{m})_{m=1}^\infty\in \cN$ by a family of DNNs with the same architecture, i.e., 
$\phi_{m}\in \cN^{N_0,N_1,\ldots, N_{L-1},N_L}_{L}$ for all $m$, where  $L, N_1,\ldots, N_{L-1}\in \N$, $N_0=d$ and $N_L=1$. Then for each $n\in \N\cup \{0\}$, there exists a DNN $\psi_n\in \cN$ such that 
\begin{align*}
[\cR_\varrho(\psi_n)](x)=\max_{m=1,\ldots, 2^n} [\cR_\varrho(\phi_m)](x), \q x\in \R^d.
\end{align*}
and $\sC(\psi_n)\le 8^n(\sC(\phi)+\f{34}{7})-\f{34}{7}$, where $\sC(\phi)$ denotes the complexity of $\phi_m$, $m\in \N$.
The same result hold for the pointwise mimimum operation.
\end{Proposition}
\begin{proof}
We first prove the result for the pointwise maximum operation by an induction on $n$. It is clear the statement holds for $n=0$ with $\psi=\phi_1$. Now for any given $k\in \N\cup\{0\}$,  we shall consider the case $2l=2^{k+1}$ by assuming the statement holds for $l=2^k$.  Note that we have
\bb\l{eq:g_2l}
g_{2l}(x)\coloneqq \max_{m=1,\ldots, 2l} f_m(x)=\max(g^{(1)}_{l}(x),g^{(2)}_{l}(x)), \q x\in \R^d,
\ee
where we denote $f_m=\cR_\varrho(\phi_m)$ for all $m$, $g^{(1)}_{l}=\max_{m=1,\ldots, l} f_m$ and $g^{(2)}_{l}=\max_{m=l+1,\ldots,2l} f_m$.
Then since $\phi_m$ has the same architecture, by induction hypothesis, we know $g^{(1)}_{l}$ and $g^{(2)}_{l}$ can be represented by networks $\psi^{(1)}$ and $\psi^{(2)}$, respectively, with the same architecture: 
$$
\psi^{(i)}_l=((W^{(m)}_1,b^{(m)}_1),(W^{(m)}_2,b^{(m)}_2),\ldots, (W^{(m)}_{L'},b^{(m)}_{L'}))\in \cN^{N'_0,N^{'}_1,\ldots, N'_{L'}}_{L'}, \q m=1,2,
$$
for some  integers  $L', N'_0,N'_1,\ldots, N'_{L}\in \N$ with $N'_0=d$ and $N'_{L'}=1$. Then one can construct the following network:
\begin{align*}
\cP&(\psi^{(1)}_l,\psi^{(2)}_l)\\
&=\bigg(\begin{pmatrix}W^{(1)}_1\\ W^{(2)}_1\end{pmatrix}, \begin{pmatrix}b^{(1)}_1\\ b^{(2)}_1\end{pmatrix}\bigg),
\bigg(\begin{pmatrix}W^{(1)}_2 & \\ & W^{(2)}_2\end{pmatrix}, \begin{pmatrix}b^{(1)}_2\\ b^{(2)}_2\end{pmatrix}\bigg),\ldots,
\bigg(\begin{pmatrix}W^{(1)}_{L'} & \\ & W^{(2)}_{L'}\end{pmatrix}, \begin{pmatrix}b^{(1)}_{L'}\\ b^{(2)}_{L'}\end{pmatrix}\bigg)\bigg),
\end{align*}
and verify that it represents the parallelization of $\psi^{(1)}$ and $\psi^{(2)}$:
$$
[\cR_\varrho(\cP(\psi^{(1)}_l,\psi^{(2)}_l))](x)=([\cR_\varrho(\psi^{(1)}_l)](x),[\cR_\varrho(\psi^{(2)}_l)](x))=(g^{(1)}_{l}(x),g^{(2)}_{l}(x)),\q x\in \R^d.
$$
Moreover, we have 
\begin{align*}
\sC(\cP(\psi^{(1)},\psi^{(2)}))&=2N^{'}_1(N'_0+1)+\sum_{i=2}^{L'}(2N^{'}_i)(2N^{'}_{i-1}+1)\\
&\le 4\bigg(N^{'}_1(N'_0+1)+\sum_{i=2}^{L'}N^{'}_i(N^{'}_{i-1}+1)\bigg)=2(\sC(\psi^{(1)}_l)+\sC(\psi^{(2)}_l)).
\end{align*}

Note that the  following identity for the max function $(x,y)\in \R^2\to \max(x,y)\in\R$:
$$
\max(x,y)=0.5\big(\max(x-y,0)+\max(y-x,0)+\max(x+y,0)-\max(-x-y,0)\big), 
$$
implies that the max function  can be represented  by the following 2-layer ReLU network:
$$
\psi_{\max}=\left(\left(\begin{pmatrix}1 & -1 \\ -1 & 1\\ 1 & 1\\ -1 &-1 \end{pmatrix}, 
0\right),
\bigg(\begin{pmatrix} 0.5 &0.5 & 0.5& -0.5\end{pmatrix}, 0\bigg)\right).
$$
Therefore, by using \eqref{eq:g_2l} and Lemma \ref{lemma:composition}, we deduce that there exists a DNN $\psi_{2l}\in \cN$ representing $g_{2l}$ with the complexity 
$$
\sC(\psi_{2l})\le 2\big(\sC(\psi_{\max})+\sC(\cP(\psi^{(1)},\psi^{(2)}))\big)=2\big(17+2(\sC(\psi^{(1)}_l)+\sC(\psi^{(2)}_l))\big).
$$
Then by using the hypothesis on  $\sC(\psi^{(1)}_l)$ and $\sC(\psi^{(2)}_l)$, we obtain that
$$
\sC(\psi_{2l})\le 34+8\bigg(8^k(\sC(\phi)+\f{34}{7})-\f{34}{7}\bigg)=8^{k+1}\bigg(\sC(\phi)+\f{34}{7}\bigg)-\f{34}{7},
$$
which completes our proof for the pointwise maximum operation. 

Finally, by  observing  the simple identity
$$
\min_{m=1,\ldots, 2^n} \cR_\varrho(\phi_m)=-\max_{m=1,\ldots, 2^n} -\cR_\varrho(\phi_m)
$$
and the fact that scaling a function can be achieved by adjusting the weights in the output layer of its DNN representation without  change its architecture, we can conclude the same result for the pointwise minimum operation.
\end{proof}

%
%

\section{Linear-implicit Euler discretizations for SDEs}\l{sec:im_euler}
In this section, we shall derive  precise error estimates of   linear-implicit Euler discretization  for  a  finite-dimensional SDE. 
In particular, we shall demonstrate that under the monotonicity condition in (H.\ref{assum:coeff_d}(\ref{assum:mono_d})), 
the approximation error of the linear-implicit Euler scheme depends  polynomially on the Lipschitz constants of the coefficients, which is crucial for our   analysis on the DNN expression rates in Section \ref{sec:proof_expression_sde}.



Let $d\in \N$ and $x_0\in \R^d$ be fixed throughout this section. We consider the following SDE:
 \bb\l{eq:sde_d}
dY_t=(-AY_t+\mu(t,Y_{t}))\,dt+ {\sigma}(t,Y_{t})\,dB_t, \q t\in (0,T];\q Y_0=x_0,
\ee
 where $(B_t)_{t\in [0,T]}$ is a $d$-dimensional Brownian motion on  a probability space   $(\Om, \cF , \bP )$. 

We now introduce two linear-implicit  Euler discretizations for  \eqref{eq:sde_d}. For any given $N\in \N$, we shall consider 
 the family of random variables     $(Y^\pi_n)_{n=0}^{N}$ defined as follows: $Y^\pi_0=(I_d+hA)^{-1}x_0$, and for all $n=0,\ldots, N-1$,
\bb\l{eq:im_euler}
Y^\pi_{n+1}-Y^\pi_n+hAY^\pi_{n+1}=h{\mu}(t_n,Y^\pi_{n})+{\sigma}(t_n,Y^\pi_{n})\Delta B_{n+1}, 
\ee
where $h=T/N$ and $\Delta B_{n+1}=B_{(n+1)h}-B_{nh}$. We shall also consider the family of random variables     $(\tilde{Y}^\pi_n)_{n=0}^{N}$ defined by the following perturbed Euler scheme: $\tilde{Y}^\pi_0=(I_d+hA)^{-1}x_0$, and for all $n=0,\ldots, N-1$,
\bb\l{eq:im_euler_p}
\tilde{Y}^\pi_{n+1}-\tilde{Y}^\pi_n+hA\tilde{Y}^\pi_{n+1}=h\tilde{\mu}(t_n,\tilde{Y}^\pi_{n})+\tilde{\sigma}(t_n,\tilde{Y}^\pi_{n})\Delta B_{n+1}.
\ee
In the sequel, we shall simply refer \eqref{eq:im_euler} and \eqref{eq:im_euler_p} as ES and PES, respectively.

We shall make the following assumptions on the coefficients of the SDE \eqref{eq:sde_d} and the Euler schemes, 
which are  analogues of (H.\ref{assum:coeff_d}) and (H.\ref{assum:perturb_d}) for the fixed $d$-dimensional problem.
\begin{Assumption}\l{assum:d}
Let $x_0\in \R^d$, $A\in \R^{d\t d}$, $\eta,D>0$ and 
$\b,\gamma,\theta,C_{\mu,0},C_{\mu,1},C_{\sigma,0},C_{\sigma,1},C_f\ge 0$ be given constants. 
Let ${\mu},\tilde{\mu}:[0,T]\t \R^d\to \R^d$, ${\sigma},\tilde{\sigma}:[0,T]\t \R^d\to \R^{d\t d}$, $f,f_D, \tilde{f}_D:\R^d\to \R$ be measurable functions satisfying the following conditions, for all $t,s\in [0,T]$ and $x,y\in \R^d$:
\bn[(a)]
\item \l{assum:d_mono}
The matrix $A$ and the functions $\mu,\sigma$ satisfy the  following monotonicity condition:
\begin{align*}
\la x-y,  {\mu}(t,x)-{\mu}(t,y)\ra +\eta\|{\mu}(t,x)-{\mu}(t,y)\|^2+&\f{1+\eta}{2}\|{\sigma}(t,x)-{\sigma}(t,y)\|^2\\
&\le \b \|x-y\|^2+ \la x-y,A(x-y)\ra.
\end{align*}
\item \l{assum:d_A}
$\la x,Ax\ra\ge 0$ for all $x\in \R^d$.
\item \l{assum:d_regu}
$\mu$ and $\sigma$ admit the following regularity: 
\begin{align*}
\|\mu(t,x)-\mu(s,y)\|\le C_{\mu,1}(\sqrt{t-s}+\|x-y\|),\q \|\mu(t,0)\|\le C_{\mu,0};\\
\|\sigma(t,x)-\sigma(s,y)\|\le C_{\sigma,1}(\sqrt{t-s}+\|x-y\|),\q \|\sigma(t,0)\|\le C_{\sigma,0}.
\end{align*}
\item \l{assum:d_approx}
 $(\mu,\sigma)$ and $(\tilde{\mu},\tilde{\sigma})$ satisfy the following estimate:
$$
\|{\mu}(t,x)-\tilde{\mu}(t,x)\|+\|{\sigma}(t,x)-\tilde{\sigma}(t,x)\|\le \gamma.
$$
\item \l{assum:d_terminal}
 $f,f_D, \tilde{f}_D$ satisfy the following estimates:
\begin{align*}
|f(x)-f_D(x)|\le C_f \|x\|^{2}1_{B_\infty(D)^c}(x), \q |f_D(x)-f_D(y)|\le C_f D \|x-y\|,\q |{f}_D(x)-\tilde{f}_D(x)|\le \theta,
\end{align*}
where $B_\infty(D)\coloneqq \{x\in \R^d\mid x_i\in [-D,D],\, \fa i\}$. 
\en
\end{Assumption}

\begin{Remark}\l{rmk:mono}
Throughout this section, we shall assume without loss of generality that $\eta\in (0,1)$ in  (H.\ref{assum:d}(\ref{assum:d_mono})). Moreover, for any given $h>0$, we can directly deduce from (H.\ref{assum:d}(\ref{assum:d_A})) that $(I_d+hA)x\not=0$ for all $x\not=0$, which implies  that  the matrix $I_d+hA$ is nonsingular, and  satisfies the estimates $\|(I_d+hA)^{-1}\|_{\textrm{op}}\le 1$ and $\|hA(I_d+hA)^{-1}\|_{\textrm{op}}\le 1$ (see e.g.~\cite[Proposition~7.2]{brezis2011}).

\end{Remark}

The following lemma estimates the growth rates of the coefficients $(\mu,\sigma)$ and $(\tilde{\mu},\tilde{\sigma})$.
\begin{Lemma}\l{lemma:growth}
Suppose (H.\ref{assum:d}) holds. Then  the functions $\mu,\sigma$  satisfy the following growth condition: for all   $\eta'\in [0,\eta)$ and $(t,x)\in [0,T]\t \R^d$, we have
\begin{align}\l{eq:mono_l2}
\la x, {\mu}(t,x)\ra +\eta'\|{\mu}(t,x)\|^2+\f{1+\eta'}{2}\|{\sigma}(t,x)\|^2\le  \a' +\bigg(\b+\f{1}{2}\bigg) \|x\|^2+\la x,Ax\ra, 
 \end{align}
 with the constant 
$$ \a'=\bigg(\f{1}{2}+\f{\eta\,\eta'}{\eta-\eta'}\bigg)C_{\mu,0}^2 +\f{(1+\eta)(1+\eta')}{2(\eta-\eta')}C_{\sigma,0}^2.
 $$
Similarly, 
 the functions $\tilde{\mu},\tilde{\sigma}$  satisfy the following growth condition: for all $(t,x)\in [0,T]\t \R^d$,
\begin{align}\l{eq:mono_l2_p}
\la x, \tilde{\mu}(t,x)\ra +\f{\eta}{2}\|\tilde{\mu}(t,x)\|^2+\f{1}{2}\|\tilde{\sigma}(t,x)\|^2\le  \f{2(1+\eta)^2}{\eta}\big(C_{\mu,0}^2+C_{\sigma,0}^2+\gamma^2\big) +(\b+1) \|x\|^2+\la x,Ax\ra. 
 \end{align}
\end{Lemma}
\begin{proof}
Note that for any given matrices $A,B\in \R^{d_1\t d_2}$, Young's inequality implies that $\|A+B\|^2\le (1+\tau)\|A\|^2+(1+1/\tau)\|B\|^2$ for all $\tau>0$.
Then for any fixed $\eta'\in (0,\eta)$, we can deduce from (H.\ref{assum:d}(\ref{assum:d_mono})) and (H.\ref{assum:d}(\ref{assum:d_regu})) that 
\begin{align}
&\la x, {\mu}(t,x)\ra +\eta'\|{\mu}(t,x)\|^2+\f{1+\eta'}{2}\|{\sigma}(t,x)\|^2\l{eq:eta'}\\
&=\la x, {\mu}(t,x)-{\mu}(t,0)+{\mu}(t,0)\ra +\eta'\|{\mu}(t,x)-{\mu}(t,0)+{\mu}(t,0)\|^2+\f{1+\eta'}{2}\|{\sigma}(t,x)-{\sigma}(t,0)+{\sigma}(t,0)\|^2\nb\\
&\le \la x, {\mu}(t,x)-{\mu}(t,0)\ra+ \f{1}{2}(\|x\|^2+\|{\mu}(t,0)\|^2) +\eta'\bigg(\f{\eta}{\eta'}\|{\mu}(t,x)-{\mu}(t,0)\|^2+\f{\eta}{\eta-\eta'}\|{\mu}(t,0)\|^2\bigg)\nb\\
&\q +\f{1+\eta'}{2}\bigg(\f{1+\eta}{1+\eta'}\|{\sigma}(t,x)-{\sigma}(t,0)\|^2+\f{1+\eta}{\eta-\eta'}\|{\sigma}(t,0)\|^2\bigg)\nb\\
&\le \b\|x\|^2+\la x,Ax\ra + \f{1}{2}\|x\|^2+\bigg(\f{1}{2}+\f{\eta\,\eta'}{\eta-\eta'}\bigg)C_{\mu,0}^2 +\f{(1+\eta)(1+\eta')}{2(\eta-\eta')}C_{\sigma,0}^2. \nb
 \end{align}
It is straightforward to verify that  the above inequality also holds for $\eta'=0$.

 On the other hand, by using (H.\ref{assum:d}(\ref{assum:d_approx})), we can obtain 
 \begin{align*}
&\la x, \tilde{\mu}(t,x)\ra +\f{\eta}{2}\|\tilde{\mu}(t,x)\|^2+\f{1}{2}\|\tilde{\sigma}(t,x)\|^2\\
&=\la x, \tilde{\mu}(t,x)-{\mu}(t,x)+{\mu}(t,x)\ra +\f{\eta}{2}\|\tilde{\mu}(t,x)-{\mu}(t,x)+{\mu}(t,x)\|^2+\f{1}{2}\|\tilde{\sigma}(t,x)-{\sigma}(t,x)+{\sigma}(t,x)\|^2\\
&\le \la x, {\mu}(t,x)\ra+ \f{1}{2}(\|x\|^2+\|\tilde{\mu}(t,x)-{\mu}(t,x)\|^2) +\f{\eta}{2}\bigg(\f{3}{2}\|{\mu}(t,x)\|^2+3\|\tilde{\mu}(t,x)-{\mu}(t,x)\|^2\bigg)\\
&\q +\f{1}{2}\bigg(\big(1+\f{3\eta}{4}\big)\|{\sigma}(t,x)\|^2+\big(1+\f{4}{3\eta}\big)\|\tilde{\sigma}(t,x)-{\sigma}(t,x)\|^2\bigg)\\
&\le \la x, {\mu}(t,x)\ra +\f{3\eta}{4}\|{\mu}(t,x)\|^2+\f{1+{3\eta}/{4}}{2}\|{\sigma}(t,x)\|^2
 + \f{1}{2}\|x\|^2 +\bigg(1+\f{3\eta}{2}+\f{2}{3\eta}\bigg)\gamma^2,
 \end{align*}
which, together with the estimate \eqref{eq:eta'} (with $\eta'=3\eta/4$), gives us that 
  \begin{align*}
&\la x, \tilde{\mu}(t,x)\ra +\f{\eta}{2}\|\tilde{\mu}(t,x)\|^2+\f{1}{2}\|\tilde{\sigma}(t,x)\|^2\\
&\le  \b\|x\|^2+\la x,Ax\ra + \f{1}{2}\|x\|^2+\bigg(\f{1}{2}+3\eta\bigg)C_{\mu,0}^2 +\f{2(1+\eta)^2}{\eta}C_{\sigma,0}^2
 + \f{1}{2}\|x\|^2 +\bigg(1+\f{3\eta}{2}+\f{2}{3\eta}\bigg)\gamma^2\\
&\le  (\b+1)\|x\|^2+\la x,Ax\ra +\f{2(1+\eta)^2}{\eta}\bigg(C_{\mu,0}^2 +C_{\sigma,0}^2+\gamma^2\bigg),
  \end{align*}
where we used  the assumption that $\eta<1$ in the last inequality.
 \end{proof}

With Lemma \ref{lemma:growth} in hand, we now present the following    moment estimate and  time regularity result for the strong solutions to \eqref{eq:sde_d}. 
Note that both the moment estimate and the regularity estimate  depend  polynomially on the parameters $\| A\|_{\rm{op}}$, $C_{\mu,l}$, $C_{\sigma,l}$, $l=0,1$. This observation plays a crucial role in our subsequent analysis.

\begin{Lemma}\l{lemma:sde_moment}
Suppose 
(H.\ref{assum:d}) holds.
Then the SDE \eqref{eq:sde_d} admits a unique strong solution $(Y_t)_{t\in [0,T]}$,
which admits the following a priori estimate: for all $p\in [2,2+\eta)$ and $t\in [0,T]$, 
\begin{align*}
\ex[\|Y_t\|^p]&\le 2^{\f{p-2}{2}}(\a_p +\|x_0\|^p)e^{p(\b+1/2) T},
\end{align*}
with the constant $\a_p$ defined as:
\bb\l{eq:d_alpha_p}
 \a_p=\bigg(\f{1}{2}+\f{\eta\,(p-2)}{\eta+2-p}\bigg)C_{\mu,0}^2 +\f{(1+\eta)(p-1)}{2(\eta+2-p)}C_{\sigma,0}^2, \q \fa p\in [2,2+\eta),
\ee
and the following time regularity: for all $t,s\in [0,T]$, 
\begin{align*}
\ex[\|Y_t-Y_s\|^2]&\le 8(1+\a +\|x_0\|^2)e^{(2\b+1) T} \bigg((t-s)^2(\| A\|^2_{\rm{op}}+C_{\mu,0}^2+C_{\mu,1}^2)+(t-s) (C_{\sigma,0}^2+C_{\sigma,1}^2)\bigg),  
\end{align*}
with $\a=\f{1}{2}C_{\mu,0}^2+\f{1+\eta}{2\eta}C_{\sigma,0}^2$.

\end{Lemma}
\begin{proof}
The \textit{a priori} estimate follows precisely the steps in the arguments  for \cite[Theorem 4.1 pp.~59]{mao1997} by applying 
It\^{o}'s formula to the quantity $(\a+\|Y_t\|^2)^{\f{p}{2}}$ and using the growth condition \eqref{eq:mono_l2} in Lemma \ref{lemma:growth}.
%
Then one can  deduce from  (H.\ref{assum:d}(\ref{assum:d_mono})) that
\begin{align*}
\ex[\|Y_t-Y_s\|^2]&\le 2 \bigg((t-s)\int_s^t\ex[\| -AY_r+\mu(s,Y_{r})\|^2]\,dr+\int_s^t \ex[\|\sigma(r,Y_r)\|^2]\,dr\bigg)\\
&\le 2 \bigg\{2(t-s)\int_s^t\bigg(\| A\|^2_{\rm{op}}\ex[\|Y_r\|^2]+2(C_{\mu,0}^2+C_{\mu,1}^2\ex[\|Y_{r}\|^2])\bigg)\,dr\\
&\,+2\int_s^t \bigg(C_{\sigma,0}^2+C_{\sigma,1}^2\ex[\|Y_{r}\|^2]\bigg)\,dr\bigg\}\\
&\le 8 \bigg((t-s)^2(\| A\|^2_{\rm{op}}+C_{\mu,0}^2+C_{\mu,1}^2)+(t-s) (C_{\sigma,0}^2+C_{\sigma,1}^2)\bigg)\bigg(1+\sup_{t\in [0,T]}\ex[\|Y_{t}\|^2]\bigg),
\end{align*}
which, along with  the estimate of $\ex[\|Y_t\|^2]$,  implies the desired modulus of continuity in time.
\end{proof}

Now we proceed to study the linear-implicit  Euler schemes \eqref{eq:im_euler} and \eqref{eq:im_euler_p}. The following proposition shows the stability of the linear-implicit Euler scheme. 

\begin{Proposition}\l{prop:L2_stable}
Suppose (H.\ref{assum:d}(\ref{assum:d_A})) holds. Let 
$h>0$ and $t\in [0,T]$ be given constants,  ${\mu}_i:[0,T]\t \R^d\to \R^d$, ${\sigma}_i:[0,T]\t \R^d\to \R^{d\t d}$, $i=1,2$, be  Borel measurable functions,
 $Z$ be  a $\R^d$-valued random variable with mean $0$ and variance $h$, and $(Y^i)_{i=1,2}$ be  $\R^d$-valued integrable  random variables  independent of $Z$. For each $i=1,2$, let $X^i$ be the  random variable defined by the following one-step linear-implicit Euler scheme: 
\bb\l{eq:onestep}
X^i-Y^i+hAX^i=h{\mu}_i(t,Y^i)+{\sigma}_i(t,Y^i)Z.
\ee
 Then we have the following stability estimate:
\begin{align*}
\f{1}{2}&\ex[\| X^1-X^2\|^2]+h \ex[\la X^1-X^2, A (X^1-X^2)\ra]\\
&\le  \f{1}{2} \ex[\| Y^1-Y^2\|^2]+h\ex\bigg[\la Y^1-Y^2, {\mu}_1(t,Y^1)-{\mu}_2(t,Y^2)\ra +\f{1}{2}h\|{\mu}_1(t,Y^1)-{\mu}_2(t,Y^2)\|^2\\
 &\q +\f{1}{2}\|{\sigma}_1(t,Y^1)-{\sigma}_2(t,Y^2)\|^2\bigg].
\end{align*}
\end{Proposition}
\begin{proof}
For notational simplicity, we introduce the following terms:  $\delta X=X^1-X^2$, $\delta Y=Y^1-Y^2$,
$\delta \mu={\mu}_1(t,Y^1)-{\mu}_2(t,Y^2)$, and $\delta \sigma={\sigma}_1(t,Y^1)-{\sigma}_2(t,Y^2)$. 
 Then we can deduce from \eqref{eq:onestep} that
$$
\delta X-\delta Y+hA (\delta X)-h (\delta \mu)=(\delta \sigma)Z.
$$
Multiplying the above identity by  $\delta X$, we obtain that 
\begin{align*}
\la \delta X, \delta X-\delta Y-h(\delta \mu) \ra&+h\la \delta X, A (\delta X)\ra =\la \delta X, (\delta \sigma)Z\ra.
\end{align*}
from which, by completing the square, one can  deduce that 
\begin{align}\l{eq:complete_sqaure}
\begin{split}
\f{1}{2}\| \delta X\|^2+\f{1}{2}\| \delta X-\delta Y-h(\delta \mu)\|^2&-\f{1}{2}\| \delta Y+h(\delta \mu)\|^2
+h\la \delta X, A (\delta X)\ra=\la \delta X, (\delta \sigma)Z\ra.
\end{split}
\end{align}
Then, by using the following inequality:
\begin{align*}
0&\le \| \delta X-\delta Y-h(\delta \mu)-(\delta \sigma)Z\|^2\\
&\le \| \delta X-\delta Y-h(\delta \mu)\|^2-2\la \delta X-\delta Y-h(\delta \mu), (\delta \sigma)Z\ra+ \|(\delta \sigma)Z\|^2,
\end{align*}
and the fact that 
$Z$ is independent of $\delta Y$,
we can obtain that
\begin{align*}
\f{1}{2}\ex[\| \delta X\|^2]&+h \ex[\la \delta X, A (\delta X)\ra]\le\f{1}{2} \ex[\| \delta Y+h(\delta \mu)\|^2]+\f{1}{2}h\ex[\|(\delta \sigma)\|^2]\\
&\le \f{1}{2} \ex[\| \delta Y\|^2]+h\ex\bigg[\la \delta Y, (\delta \mu)\ra +\f{1}{2}h\| (\delta \mu)\|^2+\f{1}{2}\|(\delta \sigma)\|^2\bigg],
\end{align*}
which completes the proof of the desired stability estimates.
\end{proof}

The next two corollaries follow directly from  Proposition \ref{prop:L2_stable}, 
which give an $L^2$-estimate of the numerical solutions to ES \eqref{eq:im_euler} and PES \eqref{eq:im_euler_p}, and establish an upper bound of the difference between these two solutions.

\begin{Corollary}\l{corollary:L2_bdd}
Suppose  (H.\ref{assum:d}) holds.
Let $N\in \N$ with $N\ge 2T/\eta$, and $(Y^\pi_n)_{n=0}^{N}$ (resp.~$(\tilde{Y}^\pi_n)_{n=0}^{N}$) be the solution to ES \eqref{eq:im_euler} (resp.~PES \eqref{eq:im_euler_p}).
Then    the  following  estimate holds:
 \begin{align*} 
\max_{n=0,\ldots, N}\bigg(\ex[\|Y^\pi_{n}\|^2]+2h\ex[\la Y^\pi_{n},  A Y^\pi_{n}\ra]\bigg)  \le 3e^{(2\b+1) T}\big(\|x_0\|^2+\a_1 T\big), \\
\max_{n=0,\ldots, N}\bigg(\ex[\|\tilde{Y}^\pi_{n}\|^2]+2h\ex[\la \tilde{Y}^\pi_{n},  A \tilde{Y}^\pi_{n}\ra]\bigg)  \le 3e^{2(\b+1) T}\big(\|x_0\|^2+\a_2 T\big),
\end{align*}
where $\a_1=\f{(1+\eta)^2}{\eta}\big(C_{\mu,0}^2+C_{\sigma,0}^2\big)$ and   $\a_2= \f{2(1+\eta)^2}{\eta}\big(C_{\mu,0}^2+C_{\sigma,0}^2+\gamma^2\big)$.
\end{Corollary}
\begin{proof}
We shall only estimate $\ex[\|Y^\pi_n\|^2]$, since the \textit{a priori} bound of $\tilde{Y}^\pi_n$ can be established by a similar approach.
For any given $n=0,\ldots, N-1$, by setting $(\mu_1,\sigma_1)=(\mu,\sigma)$, $(\mu_2,\sigma_2)=(0,0)$ and $Y^2=0$ in Proposition \ref{prop:L2_stable}, we obtain from Lemma \ref{lemma:growth} that: for all $h<\eta' \coloneqq \eta/2$,
\begin{align*}
&\ex[\|Y^\pi_{n+1}\|^2]  + 2h\ex[\la Y^\pi_{n+1},  A Y^\pi_{n+1}\ra]\\
& \le   \ex[\|Y^\pi_n\|^2]
+2h\ex\bigg[ \eta'\|{\mu}(t_n,Y^\pi_{n})\|^2+\f{1}{2}\|{\sigma}(t_n,Y^\pi_{n})\|^2+\la Y^\pi_n, {\mu}(t_n,Y^\pi_{n})\ra\bigg] \\
& \le\ex[ \|Y^\pi_n\|^2]+2h\ex\bigg[\la Y^\pi_{n},  A Y^\pi_{n}\ra+(\f{1}{2}+\eta)C_{\mu,0}^2 +\f{(1+\eta)(1+\eta/2)}{\eta}C_{\sigma,0}^2
+(\b+\f{1}{2})\|Y^\pi_{n}\|^2\bigg]\\
&\le  (1+2\b' h)\bigg(\ex[\|Y^\pi_n\|^2]+2h\ex[\la Y^\pi_{n},  A Y^\pi_{n}\ra]\bigg)+2\a' h,
\end{align*}
with $\a_1=\f{(1+\eta)^2}{\eta}\big(C_{\mu,0}^2+C_{\sigma,0}^2\big)$, $\b'=\b+1/2$, where we have used the fact $\ex[\la Y^\pi_{n},  A Y^\pi_{n}\ra]\ge 0$. Then, by  multiplying the above inequality with $(1+2\b' h)^{-(n+1)}$, we can deduce that
\begin{align*}
&(1+2\b' h)^{-(n+1)}\bigg(\ex[\|Y^\pi_{n+1}\|^2]  + 2h\ex[\la Y^\pi_{n+1},  A Y^\pi_{n+1}\ra]\bigg)\\
&\le (1+2\b' h)^{-n} \bigg(\ex[\|Y^\pi_n\|^2]+2h\ex[\la Y^\pi_{n},  A Y^\pi_{n}\ra]\bigg)+(1+2\b' h)^{-n-1}2h\a',
\end{align*}
which leads to the following estimate: for all $n=0,\ldots, N-1$,
\begin{align*}
\ex[\|Y^\pi_{n+1}\|^2]  + 2h\ex[\la Y^\pi_{n+1},  A Y^\pi_{n+1}\ra]&\le  (1+2\b' h)^{N}\bigg(\ex[\|Y^\pi_0\|^2]+2h\ex[\la Y^\pi_{0},  A Y^\pi_{0}\ra]+2\a' T\bigg)\\
&\le e^{2\b' T}\bigg(\|Y^\pi_0\|^2+2h\la Y^\pi_{0},  A Y^\pi_{0}\ra+2\a' T\bigg).
\end{align*}
We can then conclude the desired result from the Cauchy-Schwarz inequality and Remark \ref{rmk:mono}.
%
%
%
%
\end{proof}

\begin{Corollary}\l{corollary:L2_modify}
Suppose  (H.\ref{assum:d}) holds.
Let $N\in \N$ with $N\ge T\max(2\eta,1/\eta)$, and $(Y^\pi_n)_{n=0}^{N}$ (resp.~$(\tilde{Y}^\pi_n)_{n=0}^{N}$) be the solution to ES \eqref{eq:im_euler} (resp.~PES \eqref{eq:im_euler_p}). Then  we have the following error estimate:
 $$ 
\max_{n=0,\ldots, N}\bigg(\ex[\|  Y^\pi_{n}-\tilde{Y}^\pi_{n}\|^2]+2h \ex[\la  Y^\pi_{n}-\tilde{Y}^\pi_{n}, A ( Y^\pi_{n}-\tilde{Y}^\pi_{n})\ra]\bigg)\le  e^{(2\b+1)T}  T(1+\eta)\gamma^2/\eta.
$$
\end{Corollary}
\begin{proof}
Let us define the random variable $\delta Y_{n}=Y^\pi_{n}-\tilde{Y}^\pi_{n}$ for all $n$. 
For any given $n=0,\ldots, N-1$, we deduce from Proposition \ref{prop:L2_stable} that
\begin{align}
&\ex[\| \delta Y_{n+1}\|^2]+2h \ex[\la \delta Y_{n+1}, A (\delta Y_{n+1})\ra]\le   \ex[\| \delta Y_{n}\|^2]+2h\ex\bigg[\la \delta Y_{n}, {\mu}(t_n,Y^\pi_{n})-{\mu}(t_n,\tilde{Y}^\pi_{n})\ra \l{eq:L2_modify}\\
&+\la \delta Y_{n} ,{\mu}(t_n,\tilde{Y}^\pi_{n})-\tilde{\mu}(t_n,\tilde{Y}^\pi_{n})\ra  +\f{1}{2}h\|{\mu}(t_n,Y^\pi_{n})-\tilde{\mu}(t_n,\tilde{Y}^\pi_{n})\|^2 +\f{1}{2}\|{\sigma}(t_n,Y^\pi_{n})-\tilde{\sigma}(t_n,\tilde{Y}^\pi_{n})\|^2\bigg]. \nb
\end{align}

Now we estimate the last three terms in the above inequality. It is clear that the Cauchy-Schwarz inequality gives us that
\begin{align*}
\la \delta Y_{n} ,{\mu}(t_n,\tilde{Y}^\pi_{n})-\tilde{\mu}(t_n,\tilde{Y}^\pi_{n})\ra&\le \f{1}{2}(\|\delta Y_{n}\|^2+\|{\mu}(t_n,\tilde{Y}^\pi_{n})-\tilde{\mu}(t_n,\tilde{Y}^\pi_{n})\|^2),\\
\|{\mu}(t_n,{Y}^\pi_{n})-\tilde{\mu}(t_n,\tilde{Y}^\pi_{n})\|^2&=\|
{\mu}(t_n,{Y}^\pi_{n})-{\mu}(t_n,\tilde{Y}^\pi_{n})+{\mu}(t_n,\tilde{Y}^\pi_{n})-\tilde{\mu}(t_n,\tilde{Y}^\pi_{n})\|^2\\
&\le 2\|{\mu}(t_n,{Y}^\pi_{n})-{\mu}(t_n,\tilde{Y}^\pi_{n})\|^2+2\|{\mu}(t_n,\tilde{Y}^\pi_{n})-\tilde{\mu}(t_n,\tilde{Y}^\pi_{n})\|^2.
\end{align*}
Moreover, by applying the Cauchy-Schwarz inequality to Frobenius inner product of matrices, we obtain that
\begin{align*}
\|{\sigma}(t_n,Y^\pi_{n})-\tilde{\sigma}(t_n,\tilde{Y}^\pi_{n})\|^2&=\|
{\sigma}(t_n,Y^\pi_{n})-{\sigma}(t_n,\tilde{Y}^\pi_{n})+{\sigma}(t_n,\tilde{Y}^\pi_{n})-\tilde{\sigma}(t_n,\tilde{Y}^\pi_{n})\|^2\\
&\le (1+\eta)\|{\sigma}(t_n,Y^\pi_{n})-{\sigma}(t_n,\tilde{Y}^\pi_{n})\|^2+(1+\f{1}{\eta})\|{\sigma}(t_n,\tilde{Y}^\pi_{n})-\tilde{\sigma}(t_n,\tilde{Y}^\pi_{n})\|^2.
\end{align*}
Hence,  by substituting the above estimates into  \eqref{eq:L2_modify} and rearranging the terms, we deduce that
\begin{align*}
&\ex[\| \delta Y_{n+1}\|^2]+2h \ex[\la \delta Y_{n+1}, A (\delta Y_{n+1})\ra]\\
\le &\,  (1+h)  \ex[\| \delta Y_{n}\|^2]+2h\ex\bigg[\la \delta Y_{n}, {\mu}(t_n,Y^\pi_{n})-{\mu}(t_n,\tilde{Y}^\pi_{n})\ra+h\|{\mu}(t_n,Y^\pi_{n})-{\mu}(t_n,\tilde{Y}^\pi_{n})\|^2\\
&+\f{1+\eta}{2}\|{\sigma}(t_n,Y^\pi_{n})-{\sigma}(t_n,\tilde{Y}^\pi_{n})\|^2\bigg]\\
& +h\ex\bigg[(1+2h)\|{\mu}(t_n,\tilde{Y}^\pi_{n})-\tilde{\mu}(t_n,\tilde{Y}^\pi_{n})\|^2   +(1+\f{1}{\eta})\|{\sigma}(t_n,\tilde{Y}^\pi_{n})-\tilde{\sigma}(t_n,\tilde{Y}^\pi_{n})\|^2\bigg]. \nb
\end{align*}
Hence for all $N\in \N$ such that $T/N\le \min(1/(2\eta),\eta)$,   (H.\ref{assum:d}(\ref{assum:d_mono})) and (H.\ref{assum:d}(\ref{assum:d_approx})) imply that
\begin{align*}
&\ex[\| \delta Y_{n+1}\|^2]+2h \ex[\la \delta Y_{n+1}, A (\delta Y_{n+1})\ra]\\
\le &\,  (1+h)  \ex[\| \delta Y_{n}\|^2]+2h\ex[\b\| \delta Y_{n}\|^2+\la \delta Y_{n}, A (\delta Y_{n})\ra]+h(1+\eta)\gamma^2/\eta\\
\le &\,  (1+(2\b+1)h)  \ex[\| \delta Y_{n}\|^2]+2h\ex[\la \delta Y_{n}, A (\delta Y_{n})\ra]+h(1+\eta)\gamma^2/\eta.
\end{align*}
Thus, following similar arguments as those for Corollary \ref{corollary:L2_bdd}, we can conclude the desired estimate by using the fact that $\delta Y_0=0$.
\end{proof}

Now we  proceed to derive  precise error estimates of the  linear-implicit Euler schemes \eqref{eq:im_euler} and \eqref{eq:im_euler_p}. 
The following proposition shows the overall approximation error can be bounded by the one-step local truncation errors.

\begin{Proposition}\l{prop:L2_error_trun}
Suppose  (H.\ref{assum:d}) holds.
Let $N\in \N$ with $N\ge T \max \left(\f{2\b +2^{1/T}}{2^{1/T}-1},\f{1}{\eta} \right) $, $(Y_t)_{t\in [0,T]}$ be the solution to SDE \eqref{eq:sde_d}, and $(Y^\pi_n)_{n=0}^{N}$  be the solution to ES \eqref{eq:im_euler}. Then we have the following error estimate:
 \begin{align*}
&\max_{n= 0,\ldots, N}\ex[\| Y^\pi_{n}- Y_{t_{n}}\|^2]+2h\ex[\la Y^\pi_{n}- Y_{t_{n}}, A (Y^\pi_{n}- Y_{t_{n}})\ra] \\
&\le 2e^{(2\b+1)T}\bigg(\|\delta Y_0\|^2]+2h\la \delta Y_0,A\delta Y_0\ra+\f{1}{h}\sum_{n=0}^{N-1}\ex[\| e^1_{n+1}\|^2]+(1+\f{1}{\eta})\sum_{n=0}^{N-1}\ex[\|e^2_{n+1}\|^2]\bigg),
\end{align*}
with $\delta Y_0=((I_d+hA)^{-1}-I_d)x_0$, and the truncation errors  $e^1_{n+1}, e^2_{n+1}$ defined as:
\bb\l{eq:trun_error}
e^1_{n+1}\coloneqq \int_{t_n}^{t_{n+1}}\big(-A(Y_s-Y_{t_{n+1}})+\mu(s,Y_s)-{\mu}(t_n,Y_{t_{n}})\big)\,ds, \q e^2_{n+1}\coloneqq\int_{t_n}^{t_{n+1}}\big(\sigma(s,Y_s)-{\sigma}(t_n,Y_{t_{n}})\big)\,dB_s.
\ee

\end{Proposition}

\begin{proof}
For any $n=0,\ldots, N$, we define the random variables $\delta Y_{n}=Y^\pi_{n}-Y_{t_{n}}$, 
$\delta \mu_n={\mu}(t_n,Y^\pi_{n})-{\mu}(t_n,Y_{t_{n}})$, and $\delta \sigma_n={\sigma}(t_n,Y^\pi_{n})-{\sigma}(t_n,Y_{t_{n}})$. Note that we have 
\bb\l{eq:true_euler}
Y_{t_{n+1}}-Y_{t_n}+hAY_{t_{n+1}}=h{\mu}(t_n,Y_{t_{n}})+{\sigma}(t_n,Y_{t_{n}})\Delta B_{n+1}+e^1_{n+1}+e^2_{n+1}, \q 
n=0,\ldots, N-1.
\ee
Then, by subtracting \eqref{eq:true_euler} from   \eqref{eq:im_euler}, multiplying the resulting equation with $\delta Y_{n+1}$, and completing the square (cf.~\eqref{eq:complete_sqaure}), we can deduce the following  the following identity:
\begin{align*}
\f{1}{2}\|\delta Y_{n+1}\|^2+\f{1}{2}\|\delta Y_{n+1}-\delta Y_n&-h(\delta \mu_n) \|^2-\f{1}{2}\|\delta Y_n+h(\delta \mu_n) \|^2\\
&+h\la \delta Y_{n+1}, A (\delta Y_{n+1})\ra =\la \delta Y_{n+1}, (\delta \sigma_n)\Delta B_{n+1}-e^1_{n+1}-e^2_{n+1}\ra.
\end{align*}
 which, together with the following inequality:
\begin{align*}
0&\le \|\delta Y_{n+1}-\delta Y_n-h(\delta \mu_n) -(\delta \sigma_n)\Delta B_{n+1}+e^2_{n+1}\|^2\\
&=\|\delta Y_{n+1}-\delta Y_n-h(\delta \mu_n)\|^2-2\la \delta Y_{n+1}-\delta Y_n-h(\delta \mu_n) ,(\delta \sigma_n)\Delta B_{n+1}-e^2_{n+1}\ra\\
&\q +\| (\delta \sigma_n)\Delta B_{n+1}\|^2-2\la  (\delta \sigma_n)\Delta B_{n+1},e^2_{n+1}\ra +\|e^2_{n+1}\|^2,
\end{align*}
and the fact that 
$\ex[\la\delta Y_n+h(\delta \mu_n) ,(\delta \sigma_n)\Delta B_{n+1}-e^2_{n+1}\ra]=0$,
\color{black}
lead us to the estimate:
\begin{align*}
&\ex[\|\delta Y_{n+1}\|^2]+2h\ex[\la \delta Y_{n+1}, A (\delta Y_{n+1})\ra] \\
&\le \ex[\|\delta Y_n+h(\delta \mu_n) \|^2]-2\ex[\la \delta Y_{n+1}, e^1_{n+1}\ra]\\
&\q + \ex[\| (\delta \sigma_n)\Delta B_{n+1}\|^2]-2\ex[\la  (\delta \sigma_n)\Delta B_{n+1},e^2_{n+1}\ra] +\ex[\|e^2_{n+1}\|^2]\\
&\le \ex[\|\delta Y_n\|^2+2h\la \delta Y_n,\delta \mu_n \ra+h^2\|\delta \mu_n \|^2]+h\ex[\| \delta Y_{n+1}\|^2]+\f{1}{h}\ex[\| e^1_{n+1}\|^2]\\
&\q + \ex[\| (\delta \sigma_n)\Delta B_{n+1}\|^2]+\eta\ex[\|(\delta \sigma_n)\Delta B_{n+1}\|^2]+\f{1}{\eta}\ex[\|e^2_{n+1}\|^2] +\ex[\|e^2_{n+1}\|^2].
\end{align*}
Consequently, for any $h\le \eta$, we have
\begin{align*}
&(1-h)\ex[\|\delta Y_{n+1}\|^2]+2h\ex[\la \delta Y_{n+1}, A (\delta Y_{n+1})\ra] \\
&\le \ex[\|\delta Y_n\|^2]+2h\ex\bigg[\la \delta Y_n,\delta \mu_n \ra+\eta\|\delta \mu_n \|^2+\f{1+\eta}{2}\|\delta \sigma_n\|^2\bigg]+\f{1}{h}\ex[\| e^1_{n+1}\|^2]+(1+\f{1}{\eta})\ex[\|e^2_{n+1}\|^2]\\
&\le (1+2\b h)\ex[\|\delta Y_n\|^2]+2h\ex[\la \delta Y_n,A\delta Y_n\ra]+\f{1}{h}\ex[\| e^1_{n+1}\|^2]+(1+\f{1}{\eta})\ex[\|e^2_{n+1}\|^2],
\end{align*}
from which, one can deduce by induction that
\begin{align*}
&\ex[\|\delta Y_{n+1}\|^2]+2h\ex[\la \delta Y_{n+1}, A (\delta Y_{n+1})\ra] \\
&\le \bigg(\f{1+2\b h}{1-h}\bigg)^N\bigg(\ex[\|\delta Y_0\|^2]+2h\ex[\la \delta Y_0,A\delta Y_0\ra]+\f{1}{h}\sum_{n=0}^{N-1}\ex[\| e^1_{n+1}\|^2]+(1+\f{1}{\eta})\sum_{n=0}^{N-1}\ex[\|e^2_{n+1}\|^2]\bigg),
\end{align*}
which leads to  the desired statement for all large enough $N$ such that $(\f{N+2\b T}{N-T})^T\le 2$.
\end{proof}

Now we are ready to present the  strong convergence result  of the perturbed Euler scheme.

\begin{Theorem}\l{thm:L2_error}
Suppose  (H.\ref{assum:d}) holds.
Let $N\in \N$ with $N\ge T \max \left(\f{2\b +2^{1/T}}{2^{1/T}-1},\f{1}{\eta},2\eta \right) $, $(Y_t)_{t\in [0,T]}$ be the solution to SDE \eqref{eq:sde_d}, and $(\tilde{Y}^\pi_n)_{n=0}^{N}$  be the solution to PES \eqref{eq:im_euler_p}. Then we have the following error estimate:
\begin{align*}
&\max_{n= 0,\ldots, N}\ex[\| \tilde{Y}^\pi_{n}- Y_{t_{n}}\|^2] \\
&\,\le 
C_{(\b,\eta,T)}h\bigg\{
\bigg(1+\| A\|^2_{\rm{op}}+\sum_{l=0}^1(C_{\mu,l}^2+C_{\sigma,l}^2)\bigg)^2\|x_0\|^2
+\bigg(1+\| A\|^2_{\rm{op}}+\sum_{l=0}^1(C_{\mu,l}^2+C_{\sigma,l}^2)\bigg)^3+\f{{\gamma^2}}{h}\bigg\},
\end{align*}

\end{Theorem}

\begin{proof}
For any given $n=0,\ldots, N-1$, 
by using (H.\ref{assum:d}(\ref{assum:d_regu})) and the Cauchy-Schwarz inequality,
we can estimate the truncation error   $e^1_{n+1}$ defined by \eqref{eq:trun_error} as follows:
\begin{align*}
\ex[\|e^1_{n+1}\|^2]&\le \ex\bigg[\bigg( \int_{t_n}^{t_{n+1}}\|-A(Y_s-Y_{t_{n+1}})+\mu(s,Y_s)-{\mu}(t_n,Y_{t_{n}})\|\,ds\bigg)^2\bigg]\\
&\le 2h \ex\bigg[ \int_{t_n}^{t_{n+1}}\bigg(\|A(Y_s-Y_{t_{n+1}})\|^2+\|\mu(s,Y_s)-{\mu}(t_n,Y_{t_{n}})\|^2\bigg)\,ds\bigg]\\
&\le 2h \ex\bigg[\int_{t_n}^{t_{n+1}}\bigg(\|A\|^2_{\rm{op}}\|Y_{t_{n+1}}-Y_s\|^2+C_{\mu,1}^2(\sqrt{s-t_n}+\|Y_s-Y_{t_{n}}\|)^2\bigg)\,ds\bigg]\\
&\le 2h^2\bigg[(\|A\|^2_{\rm{op}}+2C_{\mu,1}^2)\sup_{s\in [t_n,t_{n+1}]}\bigg(\ex[\|Y_{t_{n+1}}-Y_s\|^2]+\ex[\|Y_s-Y_{t_{n}}\|^2]\bigg)+2C_{\mu,1}^2h\bigg].
\end{align*}
Similarly, one can obtain the following upper bound of the  truncation error   $e^2_{n+1}$:
\begin{align*}
\ex[\|e^2_{n+1}\|^2]&=\ex\bigg[\int_{t_n}^{t_{n+1}}\|\sigma(s,Y_s)-{\sigma}(t_n,Y_{t_{n}})\|^2\,ds\bigg]\le \ex\bigg[\int_{t_n}^{t_{n+1}}C_{\sigma,1}^2(\sqrt{s-t_n}+\|Y_s-Y_{t_{n}}\|)^2\,ds\bigg]\\
&\le 2hC_{\sigma,1}^2 \bigg[h+\sup_{s\in [t_n,t_{n+1}]}\ex[\|Y_s-Y_{t_{n}}\|^2]\bigg].
\end{align*}
Thus, if we denote by $(Y^\pi_n)_{n=0}^{N}$   the solution to ES \eqref{eq:im_euler}, then for all 
$N\in \N$ with $N\ge T \max \left(\f{2\b +2^{1/T}}{2^{1/T}-1},\f{1}{\eta} \right) $, 
we can infer from Proposition \ref{prop:L2_error_trun} and the assumption $\eta<1$ that
 \begin{align*}
&\max_{n= 0,\ldots, N}\ex[\| Y^\pi_{n}- Y_{t_{n}}\|^2]+2h\ex[\la Y^\pi_{n}- Y_{t_{n}}, A (Y^\pi_{n}- Y_{t_{n}})\ra] \\
&\le 2e^{(2\b+1)T}\bigg(\|\delta Y_0\|^2+2h\la \delta Y_0,A\delta Y_0\ra+\f{1}{h}\sum_{n=0}^{N-1}\ex[\| e^1_{n+1}\|^2]+(1+\f{1}{\eta})\sum_{n=0}^{N-1}\ex[\|e^2_{n+1}\|^2]\bigg)\\
&\le 2e^{(2\b+1)T}\bigg(\|\delta Y_0\|^2+2h\la \delta Y_0,A\delta Y_0\ra\\
&\qq +4T\bigg[C_{\mu,1}^2h+(\|A\|^2_{\rm{op}}+C_{\mu,1}^2)\sup_{s\in [t_n,t_{n+1}]}\bigg(\ex[\|Y_{t_{n+1}}-Y_s\|^2]+\ex[\|Y_s-Y_{t_{n}}\|^2]\bigg)\bigg]\\
&\qq +(1+\f{1}{\eta})2TC_{\sigma,1}^2 \bigg[h+\sup_{s\in [t_n,t_{n+1}]}\ex[\|Y_s-Y_{t_{n}}\|^2]\bigg]\bigg)\\
&\le C_{(\b,\eta,T)}\bigg[\la \delta Y_0,(I_d+hA)\delta Y_0\ra+(C_{\mu,1}^2+C_{\sigma,1}^2)h\\
&\qq +(\|A\|^2_{\rm{op}}+C_{\mu,1}^2+C_{\sigma,1}^2)\sup_{s\in [t_n,t_{n+1}]}\bigg(\ex[\|Y_{t_{n+1}}-Y_s\|^2]+\ex[\|Y_s-Y_{t_{n}}\|^2]\bigg)\bigg],
\end{align*}
where $\delta Y_0=((I_d+hA)^{-1}-I_d)x_0$.
Then, we can directly deduce the following inequality from Remark \ref{rmk:mono}:
\begin{align*}
\la \delta Y_0,(I_d+hA)\delta Y_0\ra=\la -hA(I_d+hA)^{-1}x_0,-hAx_0\ra\le \|hA(I_d+hA)^{-1}x_0\|\|hAx_0\|\le h\|A\|_{\rm{op}}\|x_0\|^2,
\end{align*}
which, together with $h<1$ and the time regularity of the solution $(Y_t)_t$ (Lemma \ref{lemma:sde_moment}), leads us to
 \begin{align*}
&\max_{n= 0,\ldots, N}\ex[\| Y^\pi_{n}- Y_{t_{n}}\|^2] \le C_{(\b,\eta,T)}\bigg[h\|A\|_{\rm{op}}\|x_0\|^2+(C_{\mu,1}^2+C_{\sigma,1}^2)h\\
&\, +(\|A\|^2_{\rm{op}}+C_{\mu,1}^2+C_{\sigma,1}^2)
(1+C_{\mu,0}^2+C_{\sigma,0}^2+\|x_0\|^2) \bigg(h^2(\| A\|^2_{\rm{op}}+C_{\mu,0}^2+C_{\mu,1}^2)+h(C_{\sigma,0}^2+C_{\sigma,1}^2)\bigg)\bigg]\\
&\,\le 
C_{(\b,\eta,T)}h\bigg\{
\bigg[ \|A\|_{\rm{op}}+\bigg(\| A\|^2_{\rm{op}}+\sum_{l=0}^1(C_{\mu,l}^2+C_{\sigma,l}^2)\bigg)^2\bigg]\|x_0\|^2\\
&\,+\bigg[ C_{\mu,1}^2+C_{\sigma,1}^2+\bigg(1+C_{\mu,0}^2+C_{\sigma,0}^2\bigg)\bigg(\| A\|^2_{\rm{op}}+\sum_{l=0}^1(C_{\mu,l}^2+C_{\sigma,l}^2)\bigg)^2\bigg]\bigg\}\\
&\,\le 
C_{(\b,\eta,T)}h\bigg\{
\bigg(1+\| A\|^2_{\rm{op}}+\sum_{l=0}^1(C_{\mu,l}^2+C_{\sigma,l}^2)\bigg)^2\|x_0\|^2
+\bigg(1+\| A\|^2_{\rm{op}}+\sum_{l=0}^1(C_{\mu,l}^2+C_{\sigma,l}^2)\bigg)^3\bigg\}.
\end{align*}
Finally, 
by further assuming  $N\ge T\max(2\eta,1/\eta)$, and using Corollary \ref{corollary:L2_modify}, we can conclude that:
\begin{align*}
&\max_{n= 0,\ldots, N}\ex[\| \tilde{Y}^\pi_{n}- Y_{t_{n}}\|^2] \le 
2\max_{n= 0,\ldots, N}\bigg(\ex[\| {Y}^\pi_{n}- Y_{t_{n}}\|^2]+\ex[\| \tilde{Y}^\pi_{n}- Y^\pi_{t_{n}}\|^2]\bigg)\\
&\,\le 
C_{(\b,\eta,T)}h\bigg\{
\bigg(1+\| A\|^2_{\rm{op}}+\sum_{l=0}^1(C_{\mu,l}^2+C_{\sigma,l}^2)\bigg)^2\|x_0\|^2
+\bigg(1+\| A\|^2_{\rm{op}}+\sum_{l=0}^1(C_{\mu,l}^2+C_{\sigma,l}^2)\bigg)^3+\f{{\gamma^2}}{h}\bigg\},
\end{align*}
which completes the proof of the desired error estimate.
 \end{proof}

We end this section with the following weak convergence rate of the perturbed Euler scheme \eqref{eq:im_euler_p} with a perturbed terminal cost.

\begin{Theorem}\l{thm:weak}
Suppose  (H.\ref{assum:d}) holds. 
Let $N\in \N$ with $N\ge T \max \left(\f{2\b +2^{1/T}}{2^{1/T}-1},\f{1}{\eta},2\eta \right) $, $(Y_t)_{t\in [0,T]}$ be the solution to SDE \eqref{eq:sde_d}, and $(\tilde{Y}^\pi_n)_{n=0}^{N}$  be the solution to PES \eqref{eq:im_euler_p}. 
  Then we have the following error estimate:
\begin{align*}
&\max_{n=0,\ldots, N}|\ex[f(Y_{t_n})-\ex[\tilde{f}_D(\tilde{Y}^\pi_{n})]|
\le C_{(\b,\eta,T)} C_f\bigg\{
 D^{\f{-2\eta}{\eta+4}}(\|x_0\|^{2+\eta/2}+\|x_0\|^2+C_{\mu,0}^2+C_{\sigma,0}^2)\\
 &+Dh^{\f{1}{2}}\bigg[
\bigg(1+\| A\|^2_{\rm{op}}+\sum_{l=0}^1(C_{\mu,l}^2+C_{\sigma,l}^2)\bigg)\|x_0\|\ +\bigg(1+\| A\|^2_{\rm{op}}+\sum_{l=0}^1(C_{\mu,l}^2+C_{\sigma,l}^2)\bigg)^{\f{3}{2}}+{\gamma}h^{-\f{1}{2}}\bigg]\bigg\}+\theta.
\end{align*}

\end{Theorem}

\begin{proof}
The assumption (H.\ref{assum:d}(\ref{assum:d_terminal})) implies that for all $n=0,\ldots, N$,
\begin{align}
|\ex[f(Y_{t_n})-\ex[\tilde{f}_D(\tilde{Y}^\pi_{n})]|&\le \ex[|f(Y_{t_n})-f_D(Y_{t_n})|]+\ex[|f_D(Y_{t_n})-f_D(\tilde{Y}^\pi_n)|]+\ex[|f_D(\tilde{Y}^\pi_n)]-\tilde{f}_D(\tilde{Y}^\pi_{n})|]\nb\\
&\le C_f\ex[\|Y_{t_n}\|^2 1_{B_\infty(D)^c}(Y_{t_n})]+C_fD\ex[\|Y_{t_n}-\tilde{Y}^\pi_n\|]+\theta. \l{eq:weak1}
\end{align}

Now we bound the term $\ex[\|Y_{t_n}\|^2 1_{B_\infty(D)^c}(Y_{t_n})]$. 
Since $\{\|x\|\le D\}\subset B_\infty(D)$, we have $1_{B_\infty(D)^c}(Y_{t_n})\le 1_{\{\|x\|\ge D\}}(Y_{t_n})$. Thus 
for any given $p'\in (1,1+\eta/2)$, by using H\"{o}lder's inequality and Lemma \ref{lemma:sde_moment}, we obtain that
\begin{align*}
\ex[\|Y_{t_n}\|^21_{B_\infty(D)^c}(Y_{t_n})]
&\le \ex[\|Y_{t_n}\|^{2 p'}]^{\f{1}{p'}}\ex[1_{\{\|x\|\ge D\}}(Y_{t_n})]^{\f{p'-1}{p'}}\le \ex[\|Y_{t_n}\|^{2 p'}]^{\f{1}{p'}}\bigg( \f{\ex[\|Y_{t_n}\|^2]}{D^2}\bigg)^{\f{p'-1}{p'}}\\
&\le  \bigg(2^{ p'-1}(\a_{2 p'} +\|x_0\|^{2 p'})e^{2 p'(\b+1/2) T}\bigg)^{\f{1}{p'}} D^{\f{-2(p'-1)}{p'}}\bigg((\a_2 +\|x_0\|^2)e^{2(\b+1/2) T} \bigg)^{\f{p'-1}{p'}}
\end{align*}
with the constant  $\a_p$ defined as in \eqref{eq:d_alpha_p} for all $p\in [2,2+\eta)$. Thus  by choosing  $p'=1+\eta/4$, we deduce from Young's inequality $xy\le \f{1}{p}x^p+\f{1}{q}y^q$, $x,y\ge 0$, $p>1$, $q=p/(p-1)$, that
\begin{align*}
\ex[\|Y_{t_n}\|^21_{B_\infty(D)^c}(Y_{t_n})]&\le C_{(\b,\eta,T)} 
 D^{\f{-2(p'-1)}{p'}}(\a_{2 p'} +\|x_0\|^{2 p'})^{\f{1}{p'}} (\a_2 +\|x_0\|^2)^{\f{p'-1}{p'}}\\
 &\le C_{(\b,\eta,T)} 
 D^{\f{-2(p'-1)}{p'}}\bigg(\a_{2 p'} +\|x_0\|^{2 p'}+ \a_2 +\|x_0\|^2\bigg)\\
&\le  C_{(\b,\eta,T)} 
 D^{\f{-2(p'-1)}{p'}}(C_{\mu,0}^2+C_{\sigma,0}^2+\|x_0\|^{2 p'}+\|x_0\|^2).
\end{align*}
Thus, by using \eqref{eq:weak1} and Theorem \ref{thm:L2_error}, we obtain that 
\begin{align*}
&|\ex[f(Y_{t_n})-\ex[\tilde{f}_D(\tilde{Y}^\pi_{n})]|\le C_f\ex[\|Y_{t_n}\|^2 1_{B_\infty(D)^c}(Y_{t_n})]+C_fD\ex[\|Y_{t_n}-\tilde{Y}^\pi_n\|^2]^{1/2}+\theta\\
&\le C_{(\b,\eta,T)} C_f\bigg\{
 D^{\f{-2\eta}{\eta+4}}(\|x_0\|^{2+\eta/2}+\|x_0\|^2+C_{\mu,0}^2+C_{\sigma,0}^2)+D
h^{\f{1}{2}}\bigg[
\bigg(1+\| A\|^2_{\rm{op}}+\sum_{l=0}^1(C_{\mu,l}^2+C_{\sigma,l}^2)\bigg)^2\|x_0\|^2\\
&\qq +\bigg(1+\| A\|^2_{\rm{op}}+\sum_{l=0}^1(C_{\mu,l}^2+C_{\sigma,l}^2)\bigg)^3+\f{\gamma^2}{h}\bigg]^{\f{1}{2}}\bigg\}+\theta,
\end{align*}
which, along with the fact that $\psi(x)=x^{1/2}$ is subadditive on $[0,\infty)$, completes our proof.
\end{proof}

\section{Linear-implicit Euler discretizations for controlled SDEs}

In this section, we extend 
the convergence analysis in Section \ref{sec:im_euler}
to 
SDEs controlled by a piecewise-constant deterministic strategy,
whose 
coefficients are merely piecewise 
H\"{o}lder continuous in time.
We shall establish that, similar to Theorem \ref{thm:weak},
the approximation error of the perturbed Euler scheme 
depends polynomially on the Lipschitz constant of the coefficients.
Such error estimate will be used in Section \ref{sec:proof_expression_sde_ctrl}
to establish the expression rate of DNN for  value functions of zero-sum games.

We start by introducing the controlled SDE and its Euler approximations.
Let $d,m,M\in \N$ and $x_0\in \R^d$
 be fixed. We consider the following SDE: $Y_0=x_0$,
 \bb\l{eq:sde_d_ctrl}
dY_t=(-AY_t+\mu(t,Y_{t},u_k))\,dt+ {\sigma}(t,Y_{t},u_k)\,dB_t, \q t\in [\bar{t}_k,\bar{t}_{k+1}),\, k=0,\ldots,M-1,
\ee
 where 
 $\bar{t}_k=kT/M$, $k=0,\ldots,M$,
  $U=\{u_k\}_{k=0}^{M-1}\subset \R^m$
  is the set of control parameters,
 $\mu:[0,T]\t \R^d\t U \to \R^d$ 
 and  $\sigma:[0,T]\t \R^d\t U\to \R^{d\t d}$
 are given functions,
  and $(B_t)_{t\in [0,T]}$ is a $d$-dimensional Brownian motion on  a probability space   $(\Om, \cF , \bP )$. 

Now we introduce 
two Euler schemes for \eqref{eq:sde_d_ctrl}.
For any given $N\in \N$, 
we shall consider the   time stepsize
$h=T/(NM)$,
$t_n=nh$
for all $n=0,\ldots, NM$.
We then consider 
the  families of random variables 
$(Y_n^\pi)_{n=0}^{NM}$ 
and 
 $(\tilde{Y}^\pi_n)_{n=0}^{N}$
defined as follows: 
$Y^\pi_0=\tilde{Y}^\pi_0=(I_d+hA)^{-1}x_0$, and for all $n=0,\ldots, NM-1$,
\begin{align}
Y^\pi_{n+1}-Y^\pi_n+hAY^\pi_{n+1}&=h{\mu}(t_n,Y^\pi_{n},u_{\lfloor n/N\rfloor})+{\sigma}(t_n,Y^\pi_{n},u_{\lfloor n/N\rfloor})\Delta B_{n+1}, 
\l{eq:im_euler_ctrl}\\
\tilde{Y}^\pi_{n+1}-\tilde{Y}^\pi_n+hA\tilde{Y}^\pi_{n+1}
&=h\tilde{\mu}(t_n,\tilde{Y}^\pi_{n},u_{\lfloor n/N\rfloor})+\tilde{\sigma}(t_n,\tilde{Y}^\pi_{n},u_{\lfloor n/N\rfloor})\Delta B_{n+1},
\l{eq:im_euler_ctrl_p}
\end{align}
where $\Delta B_{n+1}=B_{t_{n+1}}-B_{t_n}$ 
and the functions $\tilde{\mu}$ and $\tilde{\sigma}$  
approximate ${\mu}$ and ${\sigma}$, respectively. 

We shall assume 
the coefficients of the SDE \eqref{eq:sde_d_ctrl} and the Euler schemes
satisfy 
(H.\ref{assum:d})  uniformly with respect to the control parameter $(u_k)_{k=0}^{M-1}$,
which 
is an  analogue of (H.\ref{assum:coeff_d_ctrl}) and (H.\ref{assum:perturb_d_ctrl}) for the fixed $d$-dimensional problem.

\begin{Assumption}\l{assum:d_ctrl}
Let $x_0\in \R^d$, $U=\{u_k\}_{k=0}^{M-1}$, $A\in \R^{d\t d}$, $\eta,D>0$ and 
$\b,\gamma,\theta,C_{\mu,0},C_{\mu,1},C_{\sigma,0},C_{\sigma,1},C_f\ge 0$ be given constants. 
Let ${\mu},\tilde{\mu}:[0,T]\t \R^d\t U\to \R^d$, 
${\sigma},\tilde{\sigma}:[0,T]\t \R^d \t U \to \R^{d\t d}$, $f,f_D, \tilde{f}_D:\R^d\to \R$ be measurable functions 
with the following properties: 
\bn[(a)]
\item 
For all $u\in U$, the matrix $A$ and 
the functions $\mu(\cdot,\cdot,u),
\tilde{\mu}(\cdot,\cdot,u):[0,T]\t \R^d\to \R^d$,
$\sigma(\cdot,\cdot,u),\tilde{\sigma}(\cdot,\cdot,u)
:[0,T]\t \R^d  \to \R^{d\t d}$ 
satisfy 
(H.\ref{assum:d}(\ref{assum:d_mono}),(\ref{assum:d_A}),(\ref{assum:d_regu}),(\ref{assum:d_approx}))
with the constants 
$\eta, \b, C_{\mu,0},C_{\mu,1},C_{\sigma,0},C_{\sigma,1},\gamma$.
\item 
The functions $f,f_D, \tilde{f}_D$ satisfy  (H.\ref{assum:d}(\ref{assum:d_terminal})) 
with the constants 
$C_f,D,\theta$. 
\en
\end{Assumption}

Note that,
by introducing the piecewise-constant function
$\tilde{u}:[0,T]\to U$ such that $\tilde{u}(t)=u_{\lfloor tM/T\rfloor}$ for all $t\in [0,T]$,
we can view 
\eqref{eq:sde_d_ctrl}-\eqref{eq:im_euler_ctrl_p}
as 
\eqref{eq:sde_d}-\eqref{eq:im_euler_p}
with coefficients:
\begin{align*}
\mu^{u}:(t,x)\in [0,T]\t \R^d\mapsto \mu(t,x,\tilde{u}(t)) \in \R^d,
&\q 
\sigma^{u}: 
(t,x)\in [0,T]\t \R^d
\mapsto \sigma(t,x,\tilde{u}(t)) \in \R^{d\t d},
\\
\tilde{\mu}^{u}:(t,x)\in [0,T]\t \R^d\mapsto \tilde{\mu}(t,x,\tilde{u}(t)) \in \R^d,
&\q
\tilde{\sigma}^{u}: 
(t,x)\in [0,T]\t \R^d
\mapsto \tilde{\sigma}(t,x,\tilde{u}(t)) \in \R^{d\t d},
\end{align*}
which, under (H.\ref{assum:d_ctrl}), satisfy
all conditions in (H.\ref{assum:d})
(with the same constants)
except the $1/2$-H\"{o}lder continuous in time on $[0,T]$. 
Consequently, we can deduce that 
Lemmas 
\ref{lemma:growth} and \ref{lemma:sde_moment},
Proposition \ref{prop:L2_stable},
 Corollaries \ref{corollary:L2_bdd} and \ref{corollary:L2_modify}
 and Proposition \ref{prop:L2_error_trun}
 (with $N$ replaced by $NM$ in the statements)
also hold for the solutions to 
\eqref{eq:sde_d_ctrl}-\eqref{eq:im_euler_ctrl_p},
whose proofs do not rely on the time regularity of coefficients.

Now we extend Theorems \ref{thm:L2_error} and \ref{thm:weak}
to establish strong and weak convergence rates
for the perturbed Euler scheme \eqref{eq:im_euler_ctrl_p}.

\begin{Theorem}\l{thm:weak_ctrl}
Suppose (H.\ref{assum:d_ctrl}) holds. 
Let $N\in \N$ with $N\ge \f{T}{M} \max \left(\f{2\b +2^{1/T}}{2^{1/T}-1},\f{1}{\eta},2\eta \right) $, $(Y_t)_{t\in [0,T]}$ be the solution to SDE \eqref{eq:sde_d_ctrl}, and $(\tilde{Y}^\pi_n)_{n=0}^{NM}$  be the solution to PES \eqref{eq:im_euler_ctrl_p}. 
  Then
 $ \max_{n= 0,\ldots, NM}\ex[\| \tilde{Y}^\pi_{n}- Y_{t_{n}}\|^2]$
 and 
 $
 \max_{n=0,\ldots, NM}|\ex[f(Y_{t_n})-\ex[\tilde{f}_D(\tilde{Y}^\pi_{n})]|
 $
 satisfy 
 the same estimates
 as in  Theorems \ref{thm:L2_error}
 and \ref{thm:weak}, respectively.  
\end{Theorem}
\begin{proof}
Let $N\in \N$ with
$N\ge \f{T}{M} \max \left(\f{2\b +2^{1/T}}{2^{1/T}-1},\f{1}{\eta},2\eta \right) $
be fixed.
Since  
Proposition \ref{prop:L2_error_trun}
also holds for 
$(Y_t)_{t\in [0,T]}$  and $(\tilde{Y}^\pi_n)_{n=0}^{NM}$,
a careful examination of the proofs of Theorems \ref{thm:L2_error} and \ref{thm:weak}
shows that 
it suffices to  estimate the local truncation errors $e^1_{n+1}$ and $e^2_{n+1}$ on $[t_n,t_{n+1}]$
for all $n=0,\ldots,NM-1$.
Note that,
for all $n=0,\ldots, NM-1$,
the definitions
of 
$(\bar{t}_k)_{k=0}^M$ and
$(t_n)_{n=0}^{NM}$
ensure that 
$[t_n,t_{n+1}]\subset 
[\bar{t}_{\lfloor n/N\rfloor},\bar{t}_{\lfloor n/N\rfloor+1}]$,
hence 
we have
$\tilde{u}(t)=u_{\lfloor n/N\rfloor}$ for $t\in [t_n,t_{n+1}]$,
which together with (H.\ref{assum:d_ctrl})
implies that 
the (controlled) coefficients
$\mu^{u}$ and  
$\sigma^{u}$ 
of \eqref{eq:sde_d_ctrl} 
are 
$1/2$-H\"{o}lder continuous 
in time on each subinterval
$[t_n,t_{n+1}]$ uniformly with respect to $n$. 
Consequently, 
it follows from the same arguments as in the proof of Theorem \ref{thm:L2_error} 
that
\begin{align*}
\ex[\|e^1_{n+1}\|^2]&\le \ex\bigg[\bigg( \int_{t_n}^{t_{n+1}}\|-A(Y_s-Y_{t_{n+1}})+\mu(s,Y_s,u_{\lfloor n/N\rfloor})-{\mu}(t_n,Y_{t_{n}},u_{\lfloor n/N\rfloor})\|\,ds\bigg)^2\bigg]\\
&\le 2h^2\bigg[(\|A\|^2_{\rm{op}}+2C_{\mu,1}^2)\sup_{s\in [t_n,t_{n+1}]}\bigg(\ex[\|Y_{t_{n+1}}-Y_s\|^2]+\ex[\|Y_s-Y_{t_{n}}\|^2]\bigg)+2C_{\mu,1}^2h\bigg],\\
\ex[\|e^2_{n+1}\|^2]&=\ex\bigg[\int_{t_n}^{t_{n+1}}\|\sigma(s,Y_s,u_{\lfloor n/N\rfloor})-{\sigma}(t_n,Y_{t_{n}},u_{\lfloor n/N\rfloor})\|^2\,ds\bigg]\\
&\le 2hC_{\sigma,1}^2 \bigg[h+\sup_{s\in [t_n,t_{n+1}]}\ex[\|Y_s-Y_{t_{n}}\|^2]\bigg].
\end{align*}
We can then proceed along the lines of Theorems 
\ref{thm:L2_error} and \ref{thm:weak}
to obtain the desired error estimates.
\end{proof}

\section{Proofs of Theorem \ref{thm:expression_sde} and Corollary \ref{cor:pde}}\l{sec:proof_expression_sde}

This section is devoted to the proofs of Theorem \ref{thm:expression_sde} and  Corollary \ref{cor:pde}.

We first prove Theorem \ref{thm:expression_sde}.
Given  $d\in \N$ and $\eps\in (0,1]$, we consider $\delta\in (0,1)$, $D>1$, $N, M\in \N$,  whose precise values will be specified later.   Let ${\kappa}=\max(1,\kappa_0,\kappa_1)$ and $(B^{m})_{m=1}^{M}$ be $M$ independent copies of   
$d$-dimensional Brownian motions defined on the same probability space $(\Om, \cF , \bP )$ supporting the Brownian motion $(B_t)_{t}$ in \eqref{eq:sde_thm_d}.\footnote{In general, suppose that $(B^{m})_{m=1}^{M}$ are defined on a
probability space $(\Om^{(M)},\cF^{(M)},\bP^{(M)})$, 
one can extend the original probability space $(\Om, \cF , \bP )$ into the product  space $(\bar{\Om}, \bar{\cF} , \bar{\bP} )=(\Om, \cF , \bP )\otimes(\Om^{(M)},\cF^{(M)},\bP^{(M)})$, and perform the subsequent analysis on the full probability space with the measure $\bar{\bP}$.}
Then, for any given $x\in \R^d$, $m=1,\ldots, M$, let  $({Y}^{x,d,m,\pi}_n)_{n=0}^{N}$ be the family of random variables      defined by the following linear-implicit Euler scheme  with perturbed coefficients 
$({\mu}^\eps_d,{\sigma}^\eps_d)$ and the $m$-th Brownian motion $(B^m_t)_t$: ${Y}^{x,d,m,\pi}_0=(I_d+hA_d)^{-1}x$, and 
\bb\l{eq:im_euler_p_d}
{Y}^{x,d,m,\pi}_{n+1}-{Y}^{x,d,m,\pi}_n+hA_d{Y}^{x,d,m,\pi}_{n+1}=h{\mu}^\eps_d(t_n,{Y}^{x,d,m,\pi}_{n})+{\sigma}^\eps_d(t_n,{Y}^{x,d,m,\pi}_{n})\Delta B^m_{n+1}, 
\ee
where $h=T/N$ and $\Delta B^m_{n+1}=B^m_{(n+1)h}-B^m_{nh}$, for all  $n=0,\ldots, N-1$. 
 
The following lemma demonstrates that there exists a realization of the perturbed  Euler scheme approximating the value function $v_d$ globally with the desired accuracy. 

\begin{Lemma}\l{lemma:random_realization}
Suppose the same assumptions of Theorem \ref{thm:expression_sde} hold.
Then it holds for some constant $c>0$, depending only on   $\b,\eta,\kappa_1,\kappa_2,\tau$ and $T$ that:
for any given $\eps\in (0,1), d\in \N$, and for any 
$D, N, M=\cO( \eps^{-c}d^c)$, $\delta=\cO(\eps^{c}d^{-c})$, there exists a realization  $\om_{\eps,d}\in \Om$, such that
\bb\l{eq:om_eps}
 \bigg(\int_{\R^d}\bigg|v_d(x)-\f{1}{M}\sum_{m=1}^Mf^\delta_{d,D}({Y}^{x,d,m,\pi}_N)(\om_{\eps,d}) \bigg|^2\, \nu_d(dx)\bigg)^{1/2}\le \eps.
 \ee
 \end{Lemma}
\begin{proof}[Proof of Lemma \ref{lemma:random_realization}]

Note that $(f^\delta_{d,D}({Y}^{x,d,m,\pi}_N))_{m=1}^M$ are independent and identically distributed random variables. Hence, by using the definition of $v_d$ and the weak uniqueness of the SDE \eqref{eq:sde_thm_d},  we can obtain that
\begin{align}
&\int_{\R^d}\ex\bigg[\bigg|\ex[f_d(Y^{x,d}_T)]-\f{1}{M}\sum_{m=1}^Mf^\delta_{d,D}({Y}^{x,d,m,\pi}_N) \bigg|^2\bigg]\, \nu_d(dx)\nb\\
&=\int_{\R^d}\big|\ex[f_d(Y^{x,d}_T)]-\ex[f^\delta_{d,D}({Y}^{x,d,1,\pi}_N)]\big|^2+\ex\bigg[\bigg|\ex[f^\delta_{d,D}({Y}^{x,d,1,\pi}_N)]-\f{1}{M}\sum_{m=1}^Mf^\delta_{d,D}({Y}^{x,d,m,\pi}_N) \bigg|^2\bigg]\, \nu_d(dx)\nb\\
&\le\int_{\R^d}\bigg(\big|\ex[f_d(Y^{x,d,1}_T)]-\ex[f^\delta_{d,D}({Y}^{x,d,1,\pi}_N)]\big|^2+\f{1}{M}\ex[|f^\delta_{d,D}({Y}^{x,d,1,\pi}_N)|^2]\bigg)\nu_d(dx), \l{eq:L2_term12}
\end{align}
where $(Y^{x,d,1}_t)_{t\in[0,T]}$ is the solution to the SDE \eqref{eq:sde_thm_d} driven by the Brownian motion $(B^1_t)_{t\in[0,T]}$.

We shall then  estimate the two terms in \eqref{eq:L2_term12} separately. Note that  Theorem \ref{thm:weak} implies that
for all $N\ge T \max \left(\f{2\b +2^{1/T}}{2^{1/T}-1},\f{1}{\eta},2\eta \right) $, we have
\begin{align*}
&\big|\ex[f_d(Y^{x,d,1}_T)]-\ex[f^\delta_{d,D}({Y}^{x,d,1,\pi}_N)]\big|\le 
 C_{(\b,\eta,T)} \kappa d^\kappa\bigg\{
 D^{\f{-2\eta}{\eta+4}}(\|x_0\|^{2+\eta/2}+\|x_0\|^2+\kappa^2 d^{2\kappa})\\
 &+Dh^{\f{1}{2}}\bigg[
\kappa^2 d^{2\kappa}\|x_0\|\ +(\kappa^2 d^{2\kappa})^{3/2}+(\delta\kappa d^{\kappa})h^{-\f{1}{2}}\bigg]\bigg\}+\delta\kappa d^{\kappa}D^{\kappa}\\
&\le C_{(\b,\eta,\kappa,T)}  d^\kappa\bigg\{d^{2\kappa}(
 D^{\f{-2\eta}{\eta+4}}+Dh^{\f{1}{2}})\|x_0\|1_{\{\|x_0\|\le 1\}}+
 d^{2\kappa}(D^{\f{-2\eta}{\eta+4}}+Dh^{\f{1}{2}})\|x_0\|^{2+\eta/2}1_{\{\|x_0\|> 1\}}\\
 &\qq+d^{3\kappa}(D^{\f{-2\eta}{\eta+4}}+Dh^{\f{1}{2}})+\delta d^{\kappa}D^{\kappa}\bigg\}.
\end{align*}
Now by letting $D^{\f{-2\eta}{\eta+4}}\le Dh^{\f{1}{2}}$, i.e., $D \ge h^{-\f{\eta+4}{6\eta+8}}$, we can deduce from the above estimate that
\begin{align*}
&\big|\ex[f_d(Y^{x,d,1}_T)]-\ex[f^\delta_{d,D}({Y}^{x,d,1,\pi}_N)]\big|\le C_{(\b,\eta,\kappa,T)} \bigg(
 d^{3\kappa}h^{\f{\eta}{3\eta+4}}\|x_0\|^{2+\f{\eta}{2}}1_{\{\|x_0\|> 1\}}+d^{4\kappa}h^{\f{\eta}{3\eta+4}}+\delta d^{2\kappa}D^{\kappa}\bigg).
\end{align*}
Therefore, by squaring the above inequality and using the integrability condition of the probability measure $\nu_d$, we obtain the following estimate:
\begin{align}
&\int_{\R^d}\big|\ex[f_d(Y^{x,d,1}_T)]-\ex[f^\delta_{d,D}({Y}^{x,d,1,\pi}_N)]\big|^2\,\nu_d(dx)\nb\\
&\le C_{(\b,\eta,\kappa,\tau,T)} \bigg(
 d^{6\kappa}h^{\f{2\eta}{3\eta+4}}\int_{\R^d}\|x_0\|^{4+\eta}\,\nu_d(dx)+d^{8\kappa}h^{\f{2\eta}{3\eta+4}}+\delta^2 d^{4\kappa}D^{2\kappa}\bigg)\nb\\
&\le C_{(\b,\eta,\kappa,\tau,T)} \bigg(
 d^{6\kappa+\max(\tau,2\kappa)}h^{\f{2\eta}{3\eta+4}}+\delta^2 d^{4\kappa}D^{2\kappa}\bigg).\l{eq:L2_term1}
\end{align}

We then proceed to obtain an upper bound of the second term in \eqref{eq:L2_term12}.
Note that (H.\ref{assum:coeff_d}(\ref{assum:terminal_d})) and (H.\ref{assum:perturb_d}(\ref {assum:perturb_d_approx})) lead to the following linear growth condition: for all $x\in \R^d$,
\begin{align*}
|f^\delta_{d,D}(x)|\le |f^\delta_{d,D}(x)-f_{d,D}(x)|+|f_{d,D}(x)-f_{d,D}(0)|+|f_{d,D}(0)|\le 
\delta\kappa d^{\kappa}D^{\kappa}+  \kappa d^{\kappa} (D \|x\|+1).
\end{align*}
Thus, we can obtain from  Corollary \ref{corollary:L2_bdd} that
\begin{align*}
\ex[|f^\delta_{d,D}({Y}^{x,d,1,\pi}_N)|^2]&\le C_{(\kappa)} (\delta^2 d^{2\kappa}D^{2\kappa}+   d^{2\kappa} +d^{2\kappa} D^2 \ex[\|{Y}^{x,d,1,\pi}_N\|^2])\\
&\le C_{(\b,\eta,\kappa,T)} (\delta^2 d^{2\kappa}D^{2\kappa}+   d^{2\kappa} +d^{2\kappa} D^2 \big(\|x\|^2+d^{2\kappa}+\delta^2 d^{2\kappa}\big)),
\end{align*}
 which, along with the following estimate 
 $$\int_{\R^d}\|x\|^2\,\nu_d(dx)\le \bigg(\int_{\R^d}\|x\|^{4+\eta}\,\nu_d(dx)\bigg)^{2/(4+\eta)}\le \tau^{1/2} d^{\tau/2},
 $$
enables us to bound the second term in \eqref{eq:L2_term12} by
 \begin{align}
\f{1}{M}\int_{\R^d}\ex[|f^\delta_{d,D}({Y}^{x,d,1,\pi}_N)|^2]\,\nu_d(dx)
&\le \f{1}{M}C_{(\b,\eta,\kappa,\tau,T)} (\delta^2 d^{2\kappa}D^{2\kappa}+   d^{2\kappa} +d^{2\kappa} D^2 \big(d^{\tau/2}+ d^{2\kappa}\big))\nb\\
&\le C_{(\b,\eta,\kappa,\tau,T)} (\delta^2 d^{2\kappa}D^{2\kappa}+   d^{2\kappa+\max(\tau/2,2\kappa)} D^2/M ). \l{eq:L2_term2}
\end{align}

Therefore, 
under the conditions  $N\ge T \max \left(\f{2\b +2^{1/T}}{2^{1/T}-1},\f{1}{\eta},2\eta \right) $ and $D=\lceil h^{-\f{\eta+4}{6\eta+8}}\rceil$,
we can deduce from the estimates \eqref{eq:L2_term12},  \eqref{eq:L2_term1} and  \eqref{eq:L2_term2} that
\begin{align*}
&\int_{\R^d}\ex\bigg[\bigg|\ex[f_d(Y^{x,d}_T)]-\f{1}{M}\sum_{m=1}^Mf^\delta_{d,D}({Y}^{x,d,m,\pi}_N) \bigg|^2\bigg]\, \nu_d(dx)\\
&\le 
C_{(\b,\eta,\kappa,\tau,T)} \bigg(
d^{6\kappa+\max(\tau,2\kappa)}h^{\f{2\eta}{3\eta+4}}+
\delta^2 d^{4\kappa}h^{-\f{(\eta+4)\kappa}{3\eta+4}}+   d^{2\kappa+\max(\tau/2,2\kappa)} h^{-\f{\eta+4}{3\eta+4}}/M\bigg).
\end{align*}
Consequently, by further assuming that
\bb\l{eq:condition}
d^{6\kappa+\max(\tau,2\kappa)}h^{\f{2\eta}{3\eta+4}}\le C\eps^2,\q 
\delta^2 d^{4\kappa} h^{-\f{(\eta+4)\kappa}{3\eta+4}}\le C\eps^2,\q 
 d^{2\kappa+\max(\tau/2,2\kappa)} h^{-\f{\eta+4}{3\eta+4}}/M\le C\eps^2,
\ee
with $C=1/(3C_{(\b,\eta,\kappa,\tau,T)})$, we have $\int_{\R^d}\ex[|v_d(x)-\f{1}{M}\sum_{m=1}^Mf^\delta_{d,D}({Y}^{x,d,m,\pi}_N) |^2]\, \nu_d(dx)<\eps^2$, which implies the existence of $\om\in\Om$ satisfying \eqref{eq:om_eps}. Finally, we complete the proof by observing that there exists a constant $c>0$, depending only on   $\b,\eta,\kappa,\tau$ and $T$, such that \eqref{eq:condition} holds for all $D, N, M=\cO( \eps^{-c}d^c)$ and $\delta=\cO( \eps^{c}d^{-c})$.
\end{proof}

We now complete the proof of Theorem \ref{thm:expression_sde}. By fixing the realization $\om_{\eps,d}\in \Om$ in Lemma \ref{lemma:random_realization}, we can see that it suffices to show that the map $x\mapsto \f{1}{M}\sum_{m=1}^Mf^\delta_{d,D}({Y}^{x,d,m,\pi}_N)(\om_{\eps,d}) $
can be represented by a neural network with the desired complexity.

We start by constructing a network for $x\mapsto f^\delta_{d,D}({Y}^{x,d,m,\pi}_N)(\om_{\eps,d})$ with a fixed $m$.
Without loss of generality, we shall assume the networks $\phi^{\mu}_{\delta, d}$ and  $(\phi^{\sigma,i}_{\delta,d})_{i=1}^d$ have more than one hidden layers, i.e., $L_{\delta,d}\ge 2$.
Note that  due to the fixed realization $\om_{\eps,d}$, the mapping $t\in [0,T]\to B^m_t(\om_{\eps,d})\in\R^d$ is a deterministic function. Hence,  
for any $n=0,\ldots, N-1$, by letting $\textbf{b}^m_{n+1}=\Delta B^{m}_{n+1}(\om_{\eps,d})\in \R^d$, we  can obtain from Lemma \ref{lemma:combination} and the fact that the networks $(\phi^{\sigma,i}_{\delta,d})_{i=1}^d$ have the same architecture (see H.\ref{assum:perturb_d}(\ref{assum:perturb_d_archi})) that  there exists a DNN 
$\phi^{\sigma,n+1}_{\delta,d}\in  \cN$, such that 
$\cL(\phi^{\sigma,n+1}_{\delta,d})=L_{\delta,d}$, $\dim(\phi^{\sigma,n+1}_{\delta,d})=(d+1,dN^{\delta,d}_1,\ldots, dN^{\delta,d}_{L_{\delta,d}-1},d)$, $\sC(\phi^{\sigma,n+1}_{\delta,d})\le d^2\sC(\phi^{\sigma,1}_{\delta,d})$ and $[\cR_\varrho(\phi^{\sigma,n+1}_{\delta,d})](t,x)={\sigma}^\delta_d(t,x)\textbf{b}^m_{n+1}$ for all $(t,x)\in [0,T]\t \R^d$.  Note that for all $n=0,\ldots, N-1$, the networks $\phi^{\sigma,n+1}_{\delta,d}$ and $\phi^{\mu}_{\delta, d}$ have the same depth, and all hidden layers of $\phi^{\sigma,n+1}_{\delta,d}$ have higher dimensions than those of the  hidden layers of $\phi^{\mu}_{\delta, d}$.

Now we consider the following inductive argument. Suppose for any  given $n=0,\ldots, N-1$, the mapping $x\mapsto {Y}^{x,d,m,\pi}_{n}(\om_{\eps,d})$ is the realization of a ReLU network $\psi_{n}^{m}\in \cN$ such that  the dimension  of its last hidden layer  satisfies $N_{\cL(\psi_{n}^{m})-1}\le 2d+(d+1)N^{\delta,d}_{L_{\delta,d}-1}$. Then we can obtain from Proposition \ref{prop:add+comp} (by letting $\phi_1=\psi_{n}^{m}, \phi_2=\phi^{\mu}_{\delta, d}, \phi_3=\phi^{\sigma,n+1}_{\delta,d}$) that there exists a network $\tilde{\psi}_{n}^{m}\in \cN$ with depth $\cL(\tilde{\psi}_{n}^{m})=\cL({\psi}_{n}^{m})+L_{\delta,d}-1$, such that for all $x\in \R^d$,
\begin{align*}
[\cR_\varrho(\tilde{\psi}_{n}^{m})](x)&=[\cR_\varrho({\psi}_{n}^{m})](x)+h[\cR_\varrho(\phi^{\mu}_{\delta, d})](t_n,[\cR_\varrho({\psi}_{n}^{m})](x))+
[\cR_\varrho(\phi^{\sigma,n+1}_{\delta,d})](t_n,[\cR_\varrho({\psi}_{n}^{m})](x))\\
&=\big({Y}^{x,d,m,\pi}_{n}+h\mu^\delta_d(t_n,{Y}^{x,d,m,\pi}_{n})+{\sigma}^\delta_d(t_n,{Y}^{x,d,m,\pi}_{n})\Delta B^m_{n+1}\big)(\om_{\eps,d}).
\end{align*}
Moreover,  the dimension of the last hidden layer of $\tilde{\psi}_{n}^{m}$ is given by $N_{\cL(\tilde{\psi}_{n}^{m})-1}= 2d+(d+1)N^{\delta,d}_{L_{\delta,d}-1}$ (see Remark \ref{rmk:add+comp}),  and the complexity satisfies
$$
\sC(\tilde{\psi}_{n}^{m})\le \sC({\psi}_{n}^{m})+4\bigg(\max(\sC(\phi^\mu_{\delta,d}),\sC(\phi^{\sigma,n+1}_{\delta,d}))+\sC(\phi^{\textnormal{Id}}_{d,2})\bigg)^3\le \sC({\psi}_{n}^{m})+4\bigg(d\sum_{i=1}^d\sC(\phi^{\sigma,i}_{\delta,d})+\sC(\phi^{\textnormal{Id}}_{d,2})\bigg)^3,
$$
where $\phi^{\textnormal{Id}}_{d,2}$  is the two-layer  representation of $d$-dimensional identity function defined as in \eqref{eq:identity}. Then, by the definition of the linear-implicit Euler scheme \eqref{eq:im_euler_p_d}, 
we know ${Y}^{x,d,m,\pi}_{n+1}(\om_{\delta,d})=(I_d+hA_d)^{-1}[\cR_\varrho(\tilde{\psi}_{n}^{m})](x)$ for all $x\in \R^d$. 
Since the architecture of a network is invariant under an affine transformation, we have shown the mapping $x\mapsto {Y}^{x,d,m,\pi}_{n+1}(\om_{\delta,d})$ is the realization of a network $\psi_{n+1}^{m}\in \cN$, which still satisfies the induction hypothesis. 
Hence, 
by observing that  ${Y}^{x,d,m,\pi}_{0}(\om_{\delta,d})=(I_d+hA_d)^{-1}[\cR_\varrho(\phi^{\textnormal{Id}}_{d,1})](x)$,
we can conclude that there exists  a network $\psi_{N}^{m}\in \cN$ representing the function $x\mapsto {Y}^{x,d,m,\pi}_{N}(\om_{\delta,d})$ with the complexity
\begin{align*}
\sC(\psi_{N}^{m})\le \sC(\phi^{\textnormal{Id}}_{d,1})+4\bigg(d\sum_{i=1}^d\sC(\phi^{\sigma,i}_{\delta,d})+\sC(\phi^{\textnormal{Id}}_{d,2})\bigg)^3(N+1).
\end{align*}
Consequently, we can infer from Lemma \ref{lemma:composition} that there exists  a network $\psi^{f,m}_{N}\in \cN$ representing the function $x\mapsto f^\delta_{d,D}({Y}^{x,d,m,\pi}_N)(\om_{\eps,d})$ with 
the complexity $\sC(\psi^{f,m}_{N})\le 2(\sC(\phi^{f}_{\delta,d,D})+\sC(\psi_{N}^{m}))$.

Finally, we observe that the Brownian path $t\mapsto B^m_t(\om^{\eps,d})$ only affects the above construction through the  vectors $(\textbf{b}^m_{n})_{n=1}^{N}$, hence the architecture of the network $\psi^{f,m}_{N}$ (i.e.~the depth and the dimensions of all layers) remains the same for each $m$.  Therefore, we can obtain from Lemma \ref{lemma:combination} that there exists a network $\psi_{\eps,d}$ with the realisation  $[\R_\varrho(\psi_{\eps,d})](x)=\f{1}{M}\sum_{m=1}^M  f^\delta_{d,D}({Y}^{x,d,m,\pi}_N)(\om_{\eps,d})$ for all $x\in \R^d$. Moreover, we can estimate the complexity of $\psi_{\eps,d}$ by using the facts $\sC(\phi^{\textnormal{Id}}_{d,1})=d^2+d$ and $\sC(\phi^{\textnormal{Id}}_{d,2})=4d^2+3d$ (see Lemma \ref{lemma:identity}), and the polynomial dependence of 
$D, N, M,\delta$ on $\eps$ and $d$ (see Lemma \ref{lemma:random_realization}):
\begin{align*}
\sC(\psi_{\eps,d})&\le M^2\sC(\psi^{f,m}_{N})\le 2M^2\bigg(\sC(\phi^{f}_{\delta,d,D})+\sC(\phi^{\textnormal{Id}}_{d,1})+4\bigg(d\sum_{i=1}^d\sC(\phi^{\sigma,i}_{\delta,d})+\sC(\phi^{\textnormal{Id}}_{d,2})\bigg)^3(N+1)\bigg)\\
&\le 8M^2\bigg(\kappa\delta^{-\kappa} d^{\kappa}D^{\kappa}+d^2+d+( \kappa d^{\kappa+1}\delta^{-\kappa}+4d^2+3d)^3N\bigg)\le c\eps^{-c}d^c,
\end{align*}
for some constant $c>0$, depending only on   $\b,\eta,\kappa_1,\kappa_2,\tau$ and $T$.
Hence the proof of  Theorem \ref{thm:expression_sde} is  finished.

In the remaining part of this section, we shall prove Corollary \ref{cor:pde}, which essentially follows from Theorem \ref{thm:expression_sde} and the Feynman-Kac formula in \cite[Theorem 2.2]{pardoux1998}.

\begin{proof}[Proof of Corollary \ref{cor:pde}]
Throughout this proof, 
let $d\in \N$ be a fixed natural number 
and $C$ be a generic constant, which depends on the dimension $d$ and may take a different value at each occurrence.

We first 
show the value function 
$v_d:\R^d\to \R$ 
 defined  in \eqref{eq:value_unctrl}
 satisfies $v_d(x)=u_d(0,x)$ for all $x\in \R^d$.
Note that 
(H.\ref{assum:coeff_d}(\ref{assum:regu_d}))
implies the coefficients of \eqref{eq:sde_thm_d} are Lipschitz continuous  and satisfies the estimate $\|\mu_d(t,x)\|+\|\sigma_d(t,x)\|\le C(1+\|x\|)$ for all $(t,x)\in [0,T]\t \R^d$.
For any $x,y\in \R^d$, we can also obtain from (H.\ref{assum:coeff_d}(\ref{assum:terminal_d})) 
 that
\begin{align*}
|f_d(x)|&=|f_d(x)-f_{d,1}(x)+f_{d,1}(x)-f_{d,1}(0)+f_{d,1}(0)-f_{d}(0)+f_{d}(0)|\\
&\le |f_d(x)-f_{d,1}(x)|+|f_{d,1}(x)-f_{d,1}(0)|+|f_{d,1}(0)-f_{d}(0)|+|f_{d}(0)|\\
&\le C_d^f \|x\|^21_{B_\infty(1)^c}(x)+C_d^f  \|x\|+0+C_d^f\le C(1+\|x\|^2),\\
|f_d(x)-f_d(y)|&=|f_{d,1+\|x\|_\infty+\|y\|_\infty}(x)-f_{d,1+\|x\|_\infty+\|y\|_\infty}(y)|\\
&\le C_d^f(1+\|x\|_\infty+\|y\|_\infty)\|x-y\|\le C_d^f(1+\|x\|+\|y\|)\|x-y\|,
\end{align*}
which together with
 the moment estimate of the SDE \eqref{eq:sde_thm_d} 
(see Lemma \ref{lemma:sde_moment}) 
and the definition of 
the value function 
$v_d:\R^d\to \R$ 
(see \eqref{eq:value_unctrl})
leads to the fact that 
$|v_d(x)|\le C(1+\|x\|^2)$ for all $x\in \R^d$. 
Consequently, 
the Feynman-Kac formula in \cite[Theorem 2.2]{pardoux1998}
and the uniqueness of 
continuous viscosity solution to \eqref{eq:linear_pde} with at most polynomial growth 
(see e.g.~\cite[Theorem 4.3]{jakobsen2005})
implies $v_d(x)=u_d(0,x)$ for all $x\in \R^d$.

Now let $\nu_d:A\in \cB(\R^d)\to  [0,1]$ be the probability measure on $\R^d$ defined as 
$\nu_d(A)\coloneqq\lambda_d(A\cap [0,1]^d)$, where $\lambda_d$ is the Lebesgue measure on $\R^d$. Then the desired result follows directly from Theorem \ref{thm:expression_sde}, 
and the fact that $\int_{\R^d}\|x\|^{4+\eta}\,\nu_d(dx)\le d^{2+\eta/2}$ (see \cite[Lemma 3.15]{grohs2018}).
\end{proof}

\section{Proof of Theorem \ref{thm:expression_sde_ctrl}}\l{sec:proof_expression_sde_ctrl}
This section is devoted to the proof of Theorem \ref{thm:expression_sde_ctrl}.
For each $d\in \N$, it is clear that $v_d(x)=\inf_{u_1\in \cU_{1,d}}\sup_{u_2\in \cU_{2,d}}w_d(x;u_1,u_2)$ for all $x\in \R^d$, where the function $w_d:\R^d\t \cU_{1,d}\t \cU_{2,d}\to \R$ is defined as:
$$w_d(x;u_1,u_2)=\ex\bigg[f_d(Y^{x,d,u_1,u_2}_T)+g_{d}(u_1,u_2)\bigg],\q x\in \R^d,u_1\in \cU_{1,d},u_2\in \cU_{2,d},$$
and 
 $Y^{x,d,u_1,u_2}=(Y^{x,d,u_1,u_2}_t)_{t\in [0,T]}$ is the  solution to the following $d$-dimensional controlled SDE:
\bb\l{eq:sde_d_ctrl_proof}
 dY_t=(-A_dY_t+\mu_d(t,Y_{t},u_1,u_2))\,dt+ {\sigma}_d(t,Y_{t},u_1,u_2)\,dB_t, \q t\in (0,T];\q Y_0=x.
\ee
In the following we shall first extend Theorem \ref{thm:expression_sde} to construct DNNs for the function $w_d$, and 
 then construct DNNs to represent the value function $v_d$.

Let $d\in \N$, $u_1\in \cU_{1,d},u_2\in \cU_{2,d}$ be fixed,
i.e.,
{there exists  $M\in \N$, independent of $d$ (see (H.\ref{assum:coeff_d_ctrl})),
such that 
for $i=1,2$, we have 
$u_i(t)=u_i(\bar{t}_k)\in U_{i,d}$
on $[\bar{t}_k,\bar{t}_{k+1})$,
where
$\bar{t}_k=kT/M$ for all $k=0,\ldots, M$}.
 Note that the essential steps to prove Theorem \ref{thm:expression_sde} are to study the 
 convergence order of the
 linear-implicit Euler scheme with perturbed coefficients  (see Theorem \ref{thm:weak}). 
Similarly, under the assumptions  
(H.\ref{assum:coeff_d_ctrl}) and (H.\ref{assum:perturb_d_ctrl}),
 we see
 the coefficients of 
the controlled SDE \eqref{eq:sde_d_ctrl_proof}
satisfies
(H.\ref{assum:d_ctrl}),
hence
we can conclude
from Theorem \ref{thm:weak_ctrl} that 
  the same weak convergence rate also holds for the  perturbed Euler scheme of \eqref{eq:sde_d_ctrl_proof} with a perturbed terminal cost.
%

Then, by following the same arguments as those in Section \ref{sec:proof_expression_sde}, we can deduce that 
there exists a constant $c>0$, depending only on   $\b,\eta,\kappa_1,\kappa_2,\tau$ and $T$, such that for 
any given $d\in \N$, $\delta>0$, $u_1\in \cU_{1,d},u_2\in \cU_{2,d}$, one can construct a DNN
$\psi^{u_1,u_2}_{\delta,d}\in \cN$ with $\sC(\psi^{u_1,u_2}_{\delta,d})\le cd^c \delta^{-c}$, and 
\bb\l{eq:w_d}
\bigg(\int_{\R^d} |w_d(x;u_1,u_2)-[\cR_\varrho(\psi^{u_1,u_2}_{\delta,d})](x)|^2\, \nu_d(dx)\bigg)^{1/2}<\delta.
\ee
Moreover, the family of DNNs $(\psi^{u_1,u_2}_{\delta,d})_{u_1\in \cU_{1,d},u_2\in \cU_{2,d}}$ has the same architecture (see Proposition \ref{prop:add+comp}, where the architecture of the constructed network $\psi$ does not depend on the value of $u$). 

Now suppose that  $\eps>0, d\in \N$ are given, we shall construct a DNN to represent the value function $v_d$ with an accuracy $\eps$. We consider $\delta>0$, whose value will be specified later, and construct the family of DNNs $(\psi^{u_1,u_2}_{\delta,d})_{u_1\in \cU_{1,d},u_2\in \cU_{2,d}}$ to represent the functions $(w_d(\cdot;u_1,u_2))_{u_1\in \cU_{1,d},u_2\in \cU_{2,d}}$ such that \eqref{eq:w_d} holds for each $u_1\in \cU_{1,d},u_2\in \cU_{2,d}$.
Note that  the number of intervention times $M$ is a constant independent of $d$, and the cardinality of the set $U_{d}=U_{1,d}\t U_{2,d}$ are bounded  by ${\kappa_0} d^{\kappa_0}$ (see (H.\ref{assum:coeff_d_ctrl})). Hence 
the fact that all admissible control strategies are piecewise-constant in time implies that $|\cU_{1,d}|\le ({\kappa_0} d^{\kappa_0})^M$ and $|\cU_{2,d}|\le ({\kappa_0} d^{\kappa_0})^M$.
Since $(\psi^{u_1,u_2}_{\delta,d})_{u_1\in \cU_{1,d},u_2\in \cU_{2,d}}$ have the same architecture, 
we can apply Proposition \ref{prop:min_max} twice (with $n\in \N$ such that $n\ge \log_2(|\cU_{i,d}|)$, $i=1,2$) and deduce that there exists a DNN $\psi_{\delta,d}$ such that 
$[\cR(\psi_{\delta,d})](x)=\inf_{u_1\in \cU_{1,d}}\sup_{u_2\in \cU_{2,d}} [\cR(\psi^{u_1,u_2}_{\delta,d})](x)$ for all $x\in \R^d$, and 
the complexity of $\psi_{\delta,d}$ is bounded by
$$\sC(\psi_{\delta,d})\le c\big(|\cU_{1,d}||\cU_{2,d}|\big)^3\sC(\psi^{u_1,u_2}_{\delta,d})\le cd^c \delta ^{-c},$$
 for some constant $c>0$ independent of  $d$ and $\delta$
(note that we have put the constant $\tfrac{34}{7}$ from Proposition \ref{prop:min_max} 
in the constant $c$, which is possible due to the fact that 
$\sC(\psi^{u_1,u_2}_{\delta,d})\ge 1$).
\color{black}

Finally, we specify the dependence of $\delta$  on the desired accuaracy $\eps$.
Note that the following inequality holds  for all parametrized functions $(f^{\a,\b},g^{\a,\b})_{\a\in \cA,\b\in \cB}$:
$$\bigg|\inf_{\a\in \cA}\sup_{\b\in \cB} f^{\a,\b}(x)-\inf_{\a\in \cA}\sup_{\b\in \cB} g^{\a,\b}(x)\bigg|\le \sup_{\a\in \cA,\b\in \cB}\bigg|f^{\a,\b}(x)-g^{\a,\b}(x)\bigg|, \q x\in \R^d.$$
 Thus we have
 \begin{align*}
&\int_{\R^d} |v_d(x)-[\cR_\varrho(\psi_{\delta,d})](x)|^2\, \nu_d(dx)\\
&=\int_{\R^d} \bigg|\inf_{u_1\in \cU_{1,d}}\sup_{u_2\in \cU_{2,d}}w_d(x;u_1,u_2)-\inf_{u_1\in \cU_{1,d}}\sup_{u_2\in \cU_{2,d}} [\cR(\psi^{u_1,u_2}_{\delta,d})](x)\bigg|^2\, \nu_d(dx)\\
&\le\int_{\R^d} \bigg(\sup_{(u_1,u_2)\in \cU_{1,d}\t \cU_{2,d}}\big|w_d(x;u_1,u_2)- [\cR(\psi^{u_1,u_2}_{\delta,d})](x)\big|\bigg)^2\, \nu_d(dx)\\
&=\int_{\R^d}  \bigg(\sup_{(u_1,u_2)\in \cU_{1,d}\t \cU_{2,d}}\big|w_d(x;u_1,u_2)- [\cR(\psi^{u_1,u_2}_{\delta,d})](x)\big|^2 \bigg)\, \nu_d(dx)\\
&\le \int_{\R^d}  \bigg(\sum_{(u_1,u_2)\in \cU_{1,d}\t \cU_{2,d}}\big|w_d(x;u_1,u_2)- [\cR(\psi^{u_1,u_2}_{\delta,d})](x)\big|^2 \bigg)\, \nu_d(dx)\le |\cU_{1,d}\t \cU_{2,d}|\delta^2.
 \end{align*}
 Since $|\cU_{1,d}\t \cU_{2,d}|\le (\kappa_0d^{\kappa_0})^M$, by choosing $\delta\le \eps(\kappa_0d^{\kappa_0})^{-M/2}$, 
 one can construct a DNN $\psi_{\eps,d}$ with the desired accuracy and complexity, and finish the proof of Theorem 
 \ref{thm:expression_sde_ctrl}.

\section{Conclusions}\l{sec:conclusion}


To the best of our knowledge, this is the first paper which rigorously explains the success of DNNs in high-dimensional  control problems with stiff systems, which arise naturally from Galerkin  approximations of controlled PDEs and SPDEs (see e.g.~\cite{farhood2008,jentzen2011,luo2012,luo2015}). The main ingredient of our proof for DNN's polynomial expression rate is that  
the underlying stochastic dynamics can be effectively described by a suitable  discrete-time system, 
whose specific realization leads us to the desired DNNs. Similar ideas can be easily extended to study optimal control problems of controlled jump diffusion processes with regime switching (see e.g.~\cite{higham2006,yin2010}), which enables us to conclude that DNNs can overcome the curse of dimensionality in numerical approximations  of weakly coupled systems of   nonlocal PDEs.

Natural next steps would be to derive optimal expression rates of DNNs for control problems, and to construct DNNs for approximating   value functions in stronger norms, such as $L^p$ norms with $p>2$,  or Sobolev norms.

\appendix
\section{Basic operations of ReLU DNNs}\l{appendix}

In this section, we collect several well-known results on the representation flexibility of DNNs.

The following lemma shows a linear combination of realizations of DNNs of the same architecture is again a realization of a DNN with the same activation function, whose proof can be found  in \cite[Lemma 5.1]{jentzen2018}. 
The  result has been generalized to the case where
the DNNs have the same length but different hidden layer dimensions
in \cite[Lemma 3.9]{hutzenthaler2019}.
\color{black}

\begin{Lemma}\l{lemma:combination}
 Let $\varrho\in C(\R;\R)$, 
$L\in \N$, $M,N_0,N_1,\ldots, N_{L}\in \N$, $(\b_m)_{m=1}^M\in \R$, and 
 $(\phi_{m})_{m=1}^M\in \cN$  be DNNs such that 
 $\cL(\phi_m)={L}$ and $\dim(\phi_m)=(N_0,N_1,\ldots, N_{L-1},N_L)$ for all $m$. Then there exists $\psi\in  \cN$, such that $\cL(\psi)=L$, $\sC(\psi)\le M^2\sC(\phi_{1})$, $\dim(\phi)=(N_0,MN_1,\ldots, MN_{L-1},N_L)$ and 
 $$
[ \cR_\varrho(\psi)](x)=\sum_{m=1}^M \b_m[ \cR_\varrho(\phi_m)](x),\q x\in \R^{N_0}.
 $$
\end{Lemma}

The next result proves that  the identity function can be represented by a ReLU network, which is proved in \cite[Lemma 5.3]{elbrachter2018}.

\begin{Lemma}\l{lemma:identity}
Let   $\varrho:\R\to \R$ be the activation function defined as $\varrho(x)=\max(0,x)$ for all $x\in \R$.
For any $d,L\in \N$, consider the DNN $\phi^{\textnormal{Id}}_{d,L}\in \cN$  given by:
\bb\l{eq:identity}
\phi^{\textnormal{Id}}_{d,L}=\begin{cases}
\bigg(
\bigg(\begin{pmatrix}I_d\\ -I_d\end{pmatrix}, 0\bigg),
\underbrace{(I_{2d},0),\ldots, (I_{2d},0)}_{\textnormal{$L-2$ times}}, 
\big(\begin{pmatrix}I_{d},-I_{d} \end{pmatrix},0 \big)
 \bigg),& L\ge 2,\\
 ((I_{d},0)), &L=1.
 \end{cases}
\ee
Then we have  $[\cR_\varrho(\phi^{\textnormal{Id}}_{d,L})](x)=x$ for all $x\in \R^d$.
\end{Lemma}

Using the above representation of the identity function, one can extend a ReLU network to a network with arbitrary depth and widths of hidden layers without changing its realization. 
\begin{Lemma}\l{lemma:extension}
Let   $\varrho:\R\to \R$ be the activation function defined as $\varrho(x)=\max(0,x)$ for all $x\in \R$.
Let $L \in \N$ and $\phi\in \cN$ be a DNN with $\cL(\phi)< L$. Then there exists 
a DNN $\cE_L(\phi)\in \cN$ such that $\cL(\cE_L(\phi))=L$,  $[\cR_\varrho(\cE_L(\phi))](x)=[\cR_\varrho(\phi)](x)$ for all $x\in \R^{\dimin(\phi)}$, and 
$$
\sC(\cE_L(\phi))\le 2\big(\sC(\phi^{\textnormal{Id}}_{\dimout(\phi),L-\cL(\phi)})+\sC(\phi)\big).
$$
Moreover, let $L\in \N\cap[2,\infty)$, $N_0,N_1,\ldots, N_{L}\in \N$ and $\phi\in \cN^{N_0,N_1,\ldots, N_{L-1},N_L}_{L}$. Then for all $l\in \{1,\ldots,  L-1\}$, there exists $\cW_l(\phi)\in \cN^{N_0,N'_1,\ldots, N'_{L-1},N_L}_{L}$ such that $N'_l=N_l+1$, $N'_i=N_i$ for all $i\in \{1,\ldots,  L-1\}\setminus\{l\}$, and $\cR_\varrho(\cW_l(\phi))= \cR_\varrho(\phi)$.
\end{Lemma}
\begin{proof}
The properties of $\cE_L(\phi)$ have been proved in  \cite[Lemma 5.3]{elbrachter2018}. Now we assume $\phi=((W_i,b_i))_{i=1}^L\in \cN$, and construct the network 
$\cW(\phi)=((W'_i,b'_i))_{i=1}^L\in \cN$ by  $(W'_i,b'_i)=(W_i,b_i)$ for all $i\in \{1,\ldots,  L-1\}\setminus\{l,l+1\}$, and
\begin{align*}
W'_l&=\begin{pmatrix}W_l\\ 0 \end{pmatrix}\in \R^{(N_l+1)\t N_{l-1}},\q b'_l=\begin{pmatrix}b_l\\ 0 \end{pmatrix};\\
W'_{l+1}&=\begin{pmatrix}W_{l+1} & 0 \end{pmatrix}\in \R^{N_{l+1}\t (N_{l}+1)},\q b'_{l+1}=b_{l+1}.
\end{align*}
Note that for all $x\in \R^{N_{l-1}}$, we have $W'_lx+b'_l=\begin{pmatrix}W_lx+b_l\\ 0 \end{pmatrix}$, and for all $x\in \R^{N_{l}},y\in \R$, we have
$$
W'_{l+1}\begin{pmatrix} x \\ y\end{pmatrix}+b'_{l+1}=W_{l+1} x+b_{l+1},
$$
which implies $[\cR_\varrho(\cW_l(\phi))](x)= [\cR_\varrho(\phi)](x)$ for all $x\in \R^{N_0}$.
\end{proof}

We then recall the composition of two DNNs and the  complexity of the resulting network (see \cite[Lemma 5.3]{elbrachter2018}).
\begin{Lemma}\l{lemma:composition}
Let   $\varrho:\R\to \R$ be the activation function defined as $\varrho(x)=\max(0,x)$ for all $x\in \R$. Let $(\phi_{m})_{m=1}^2\in \cN$ be  two DNNs such that
 \begin{alignat*}{2}
\dim(\phi_m)&=(N^{(m)}_0,N^{(m)}_1,\ldots, N^{(m)}_{\cL(\phi_m)-1}, N^{(m)}_{\cL(\phi_m)}), \q m=1,2,
\end{alignat*}
and  $\dimin(\phi_1)=\dimout(\phi_2)$, i.e., $N^{(1)}_0= N^{(2)}_{\cL(\phi_2)}$. Then  there exists a DNN $\psi\in \cN$ such that $\cL(\psi)=\cL(\phi_1)+\cL(\phi_2)$, $\sC(\psi)\le 2(\sC(\phi_1)+\sC(\phi_2))$, 
and  $\cR_\varrho(\psi)(x)=[\cR_\varrho(\phi_1)]([\cR_\varrho(\phi_2)](x))$ for all $x\in \R^{\dimin(\phi_2)}$.
\end{Lemma}



\end{document}